\documentclass[a4paper]{amsart}
\numberwithin{equation}{section}
\def\RIMS{TF}
\usepackage{amssymb, amsmath}
\usepackage[dvips]{graphicx}
\newtheorem{thm}{Theorem}[section]
\newtheorem{crl}[thm]{Corollary}

\newtheorem{lmm}[thm]{Lemma}

\newtheorem{prp}[thm]{Proposition}

\theoremstyle{definition}
\newtheorem{dfn}[thm]{Definition}
\newtheorem{exa}[thm]{Example}
\theoremstyle{remark}
\newtheorem{rem}{Remark}
\usepackage[all]{xy}
\DeclareMathOperator{\ord}{ord}
\DeclareMathOperator{\IM}{Im}

\DeclareMathOperator{\ad}{ad}
\DeclareMathOperator{\trace}{trace}
\DeclareMathOperator{\rank}{rank}
\DeclareMathOperator{\Pidx}{Pidx}
\DeclareMathOperator{\idx}{idx}
\DeclareMathOperator{\End}{End}
\DeclareMathOperator{\supp}{supp}

\DeclareMathOperator{\spt}{spt}
\def\p{\partial}
\if\RIMS
\title{Classification of Fuchsian systems and their connection problem}%
\author{\textsc{Toshio Oshima}%
\footnote{Graduate School of Mathematical Sciences, %
University of Tokyo, Tokyo 153-8914, Japan%
\newline e-mail: \texttt{oshima@ms.u-tokyo.ac.jp}}}
\AuthorHead{Toshio Oshima}         
\support{Supported by Grant-in-Aid for Scientific Researches (A), %
No.\ 20244008, Japan Society of Promotion of Science}  
\VolumeNo{x}            
\YearNo{200x}           
\PagesNo{000--000}      
\communication{Received April 20, 200x.
Revised September 11, 200x.}      
\else
\title{Classification of Fuchsian systems and their connection problem}%
\author{Toshio Oshima}
\address{Graduate School of Mathematical Sciences,
University of Tokyo, 7-3-1, Komaba, Meguro-ku, Tokyo 153-8914, Japan}
\email{oshima@ms.u-tokyo.ac.jp}
\thanks{
\hspace*{12pt}Supported by Grant-in-Aid for Scientific Researches (A), 
No.\ 20244008, Japan Society of Promotion of Science\\
}
\fi
\begin{document}
%

\maketitle

\section{Introduction}
Middle convolutions introduced by Katz \cite{Kz} and extensions and 
restrictions introduced by Yokoyama \cite{Yo} give interesting operations
on Fuchsian systems on the Riemann sphere.
They are invertible and under them the solutions of the systems are 
transformed by integral transformations and the correspondence of 
their monodromy groups is concretely described (cf.~\cite{Ko4}, \cite{Ha}, 
\cite{HY}, \cite{DR2}, \cite{HF}, \cite{O2} etc.).

In this note we review the Deligne-Simpson problem, 
a combinatorial structure of middle convolutions and their relation 
to a Kac-Moody root system discovered by Crawley-Boevey \cite{CB}.  
We show with examples that middle convolutions transform the Fuchsian 
systems with a fixed number of accessory parameters into fundamental 
systems whose spectral type is in a finite set.
In \S\ref{sec:C} we give an explicit connection formula for
solutions of Fuchsian differential equations without moduli.

The author wold like to express his sincere gratitude to
Y.\ Haraoka, A.\ Kato, H.\ Ochiai, K.\ Okamoto, 
H.\ Sakai, K.\ Takemura and T.\ Yokoyama.
The discussions with them enabled the author to write this note.

\section{Tuples of partitions}
Let $\mathbf m
=\bigl(m_{j,\nu}\bigr)_{\substack{j=0,1,\ldots\\ \nu=1,2,\ldots}}$ 
be an ordered set of infinite number of non-negative integers 
indexed by non-negative integers $j$ and positive integers $\nu$.
Then $\mathbf m$ is called a \textsl{$(k+1)$-tuple of partitions of $n$} 
if the following two conditions are satisfied.
\begin{align}
  \sum_{\nu=1}^\infty m_{j,\nu}&=n\qquad(j=0,1,\ldots),
  \allowdisplaybreaks\\
  m_{j,1} &= n\qquad(j=k+1,k+2,\ldots).
\end{align}
The totality of $(k+1)$-tuples of partitions of $n$ are denoted by
${\mathcal P}_{k+1}^{(n)}$ and we put
\begin{align}
  {{\mathcal P}}_{k+1} &:=
    \bigcup_{n=0}^\infty {{\mathcal P}}_{k+1}^{(n)},\quad
  {{\mathcal P}}^{(n)} :=
    \bigcup_{k=0}^\infty {{\mathcal P}}_{k+1}^{(n)},\quad
  {{\mathcal P}}       :=
    \bigcup_{k=0}^\infty {{\mathcal P}}_{k+1},\allowdisplaybreaks\\
 \ord\mathbf m &:= n\quad\text{if \ }
 \mathbf m\in{{\mathcal P}}^{(n)},\\
 \mathbf 1&:=
  \bigl(m_{j,\nu}=\delta_{\nu,1}\bigr)_{\substack{j=0,1,\ldots\\\nu=1,2,\ldots}}\in\mathcal P^{(1)},\\
 \idx(\mathbf m,\mathbf m')&:=
 \sum_{j=0}^k\sum_{\nu=1}^\infty m_{j,\nu}m'_{j,\nu}
 -(k-1)\ord\mathbf m\cdot\ord\mathbf m'
 \quad(\mathbf m,\ \mathbf m'\in\mathcal P_{k+1}).
\end{align}
Here $\ord\mathbf m$ is called the order of $\mathbf m$.
For $\mathbf m,\,\mathbf m'\in\mathcal P$ and a non-negative integer $p$,
the tuples $p\mathbf m$ and $\mathbf m+\mathbf m'\in\mathcal P$ are 
naturally defined.
For $\mathbf m\in{{\mathcal P}}_{k+1}^{(n)}$ we choose
integers $n_0,\dots,n_k$ so that $m_{j,\nu}=0$
for $\nu>n_j$ and $j=0,\dots,k$ and we will express
$\mathbf m$ by
\begin{align*}
 \mathbf m&=(\mathbf m_0,\mathbf m_1,\dots,\mathbf m_k)\\
          &=m_{0,1},\dots,m_{0,n_0};\ldots;m_{k,1},\dots,m_{k,n_k}\\
          &=m_{0,1}\cdots m_{0,n_0},m_{1,1}\cdots m_{1,n_1},\dots,
           m_{k,1}\cdots m_{k,n_k}
\end{align*}
if there is no confusion.
Similarly $\mathbf m=(m_{0,1},\dots,m_{0,n_0})$ 
if $\mathbf m\in\mathcal P_1$. Here
\begin{equation*}
  \mathbf m_j = (m_{j,1},\dots,m_{j,n_j}) \text{ \ and \ }
  \ord\mathbf m=m_{j,1}+\cdots+m_{j,n_j}\quad(0\le j\le k).
\end{equation*}
For example $\mathbf m=(m_{j,\nu})\in{{\mathcal P}}_{3}^{(4)}$
with $m_{1,1}=3$ and
$m_{0,\nu}=m_{2,\nu}=m_{1,2}=1$ for $\nu=1,\dots,4$
will be expressed by
\begin{equation}
 \mathbf m=1,1,1,1;3,1;1,1,1,1=1111,31,1111=1^4,31,1^4.
\end{equation}
\begin{dfn}
A tuple of partition $\mathbf m\in\mathcal P$ is called \textsl{monotone} if
\begin{equation}
  m_{j,\nu} \ge m_{j,\nu+1}\quad(j=0,1,\ldots,\ \nu=1,2,\ldots)
\end{equation}
and 
$\mathbf m$ is called \textsl{indivisible} if 
the greatest common divisor of $\{m_{j,\nu}\}$ equals 1.
\end{dfn}
Let $\mathfrak S_\infty$ be the restricted permutation group of
the set of indices $\{0,1,2,3,\ldots\}=\mathbb Z_{\ge 0}$, 
which is generated by the transpositions $(j,j+1)$ with $j\in\mathbb Z_{\ge0}$.
Put $\mathfrak S_\infty':=\{\sigma\in\mathfrak S_\infty\,;\,\sigma(0)=0\}$,
which is isomorphic to $\mathfrak S_\infty$.
\begin{dfn}\label{dfn:Sinfty}
Transformation groups  $S_\infty$ and $S_\infty'$ of
$\mathcal P$ are defined by
\begin{equation}\label{eq:S_infty}
 \begin{split}
 S_\infty &:=H\ltimes S_\infty',\quad
 S_\infty' :=\prod_{j=0}^\infty G_j,\quad
 G_j\simeq\mathfrak S_\infty',\quad H\simeq\mathfrak S_\infty,\\
 m'_{j,\nu} &= m_{\sigma(j),\sigma_j(\nu)}\qquad(j=0,1,\ldots,\ \nu=1,2,\ldots)
 \end{split}
\end{equation}
for $g = (\sigma,\sigma_1,\ldots) \in S_\infty$, 
$\mathbf m=(m_{j,\nu})\in \mathcal P$ and $\mathbf m'=g\mathbf m$.
\end{dfn}
\section{Conjugacy classes of matrices}\label{S:M}
For $\mathbf m=(m_1,\dots,m_N)\in\mathcal P^{(n)}_1$ 
and $\lambda=(\lambda_1,\dots,\lambda_N)\in\mathbb C^N$
we define a matrix $L(\mathbf m;\lambda)\in M(n,\mathbb C)$ 
as follows, which is introduced and effectively used by \cite{Os}:

If $\mathbf m$ is monotone, then
\begin{equation}\begin{split}
 L(\mathbf m;\mathbf \lambda) 
  &:= \Bigl(A_{ij}\Bigr)_{\substack{1\le i\le N\\1\le j\le N}},\quad
 A_{ij}\in M(m_i,m_j,\mathbb C),\\
 A_{ij} &= \begin{cases}
          \lambda_i I_{m_i}&(i=j)\\
          I_{m_i,m_j}:=
          \Bigl(\delta_{\mu\nu}\Bigr)
          _{\substack{1\le \mu\le m_i\\1\le \nu\le m_j}}
	=
          \left(\begin{smallmatrix}
          I_{m_j} \\ 0
          \end{smallmatrix}\right)&(i=j-1)\\
          0            &(i\ne j,\ j-1)
          \end{cases}.
\end{split}
\end{equation}
Here $I_{m_i}$ denote the identity matrix of size $m_i$ and
$M(m_i,m_j,\mathbb C)$ means the set of matrices of size $m_i\times m_j$
with components in $\mathbb C$ and
$M(m,\mathbb C):=M(m,m,\mathbb C)$.

For example
\begin{equation}
 L(2,1,1;\lambda_1,\lambda_2,\lambda_3)=
 \begin{pmatrix}
 \lambda_1 & 0       & 1& 0\\
 0         &\lambda_1& 0& 0\\
 0         & 0       &\lambda_2&1\\
 0         & 0       & 0       &\lambda_3\\
 \end{pmatrix}.
\end{equation}

If $\mathbf m$ is not monotone, fix a permutation 
$\sigma$ of $\{1,\dots,N\}$ so that 
$(m_{\sigma(1)},\ldots,m_{\sigma(N)})$ is monotone and put 
$L(\mathbf m;\mathbf \lambda)=L(m_{\sigma(1)},\ldots,
m_{\sigma(N)};\lambda_{\sigma(1)},\ldots,\lambda_{\sigma(N)})$.

When $\lambda_1=\cdots=\lambda_N=\mu$, $L(\mathbf m;\lambda)$ will be
simply denoted by $L(\mathbf m,\mu)$.

We denote $A\sim B$ for $A$, $B\in M(n,\mathbb C)$ if and only if
there exists $g\in GL(n,\mathbb C)$ with $B=gAg^{-1}$.
If $A\sim L(\mathbf m;\lambda)$, 
$\mathbf m$ is called the \textsl{spectral type} of $A$
and denoted by $\spt A$. 

\begin{rem}
{\rm i)\ }
If $\mathbf m=(m_1,\dots,m_N)\in\mathcal P_1^{(n)}$ is monotone, we have 
\begin{equation}
 A\sim L(\mathbf m;\lambda)\ \Leftrightarrow\ 
  \rank\prod_{\nu=1}^k(A-\lambda_\nu)
 = n - (m_1+\cdots+m_k)\quad(k=0,1,\dots,N).
\end{equation}

{\rm ii)\ } For $\mu\in\mathbb C$ put 
\begin{equation}\label{eq:Msub}
 (\mathbf m;\lambda)_\mu
    =(m_{i_1},\ldots,m_{i_K},\mu)
\text{ \ with \ }\{i_1,\dots,i_K\}=\{i\,;\,\lambda_i=\mu\}.
\end{equation}
Then we have
\begin{equation}\label{eq:Leigen}
  L(\mathbf m;\lambda) \sim\bigoplus_{\mu\in\mathbb C}
  L\bigl((\mathbf m;\lambda)_\mu\bigr).
\end{equation}

{\rm iii)\ } Suppose $\mathbf m$ is monotone.  
Then for $\mu\in\mathbb C$ 
\begin{equation}\label{eq:LJordan}
 \begin{aligned}
  L(\mathbf m,\mu) &\sim
  \bigoplus_{j=1}^{m_1} J\bigl(\max\{\nu\,;\,m_\nu\ge j\},\mu\bigr),\\
  J(k,\mu) &:=L(1^k,\mu)\in M(k,\mathbb C).&\text{(Jordan cell)}
 \end{aligned}
\end{equation}

{\rm iv)\ }  For $A\in M(n,\mathbb C)$ we put $Z_{M(n,\mathbb C)}(A)
:=\{X\in M(n,\mathbb C)\,;\,AX=XA\}$.  Then 
\begin{equation}\label{eq:cent}
  \dim Z_{M(n,\mathbb C)}\bigl(L(\mathbf m;\lambda)\bigr)
  = m_1^2+m_2^2+\cdots.
\end{equation} 
\end{rem}
Note that the Jordan canonical form of $L(\mathbf m;\lambda)$ is 
easily obtained by \eqref{eq:Leigen} and \eqref{eq:LJordan}.
For example $L(2,1,1,\mu)\sim J(3,\mu)\oplus J(1,\mu)$.
\begin{lmm}\label{lmm:conj}
Let $A(t)$ be a continuous map of $[0,1)$ to $M(n,\mathbb C)$.
Suppose there exists a partition $\mathbf m=(m_1,\dots,m_N)$ of 
$n$ and continuous function $\lambda(t)
$ of $(0,1)$ to 
$\mathbb C^N$  so that
$ 
  A(t)\sim L\bigl(\mathbf m;\lambda(t)\bigr)
$ 
for any $t\in (0,1)$.
If $\dim Z_{M(n,\mathbb C)}\bigl(A(t)\bigr)$ is constant for $t\in[0,1)$, then
$A(0)\sim L\bigl(\mathbf m;\lim_{t\to0}\lambda(t)\bigr)$.
\end{lmm}
\begin{proof}
The proof is reduced to the result (cf.~Remark~20) in \cite{Os}
but a more elementary proof will be given.
First note that $\lim_{t\to0}\lambda(t)$ exists.

We may assume that $\mathbf m$ is monotone.
Fix $\mu\in\mathbb C$ and put
$\{i_1,\dots,i_K\}=\{i;\,\lambda_i(0)=\mu\}$
with $1\le i_1< i_2<\cdots< i_K\le N$.
Then
\[
 \rank\bigl(A(0)-\mu\bigr)^k\le 
 \rank\prod_{\nu=1}^k
 \bigl(A(t)-\lambda_{i_\nu}(t)\bigr)=
 n-(m_{i_1}+\cdots+m_{i_k}).
\]
Putting $m'_{i_k}=\rank\bigl(A(0)-\mu\bigr)^{k-1}
 - \rank\bigl(A(0)-\mu\bigr)^k$, we have
\begin{gather*}
  m_{i_1}\ge m_{i_2}\ge\cdots\ge m_{i_K}>0,\quad
  m'_{i_1}\ge m'_{i_2}\ge\cdots\ge m'_{i_K}\ge 0,\\
  m_{i_1}+\cdots+m_{i_k}\le m'_{i_1}+\cdots+m'_{i_k}\quad(k=1,\dots,K).
\end{gather*}
Then the following lemma and the equality $\sum m_i^2=\sum (m'_i)^2$
imply $m_i=m'_i$. 
\end{proof}
\begin{lmm}
Let $\mathbf m$ and $\mathbf m'\in\mathcal P_1$ be monotone partitions
satisfying
\begin{equation}\label{eq:LB}
  m_1+\cdots+m_j\le m_1'+\cdots+m'_j\quad(j=1,2,\ldots).
\end{equation}
If $\mathbf m\ne\mathbf m'$, then
\begin{equation*}
 \sum_{j=1}^\infty m_j^2 < \sum_{j=1}^\infty (m_j')^2.
\end{equation*}
\end{lmm}
\begin{proof}
Let $K$ be the largest integer with $m_K\ne 0$ and $p$ be the smallest 
integer $j$ such that the inequality in \eqref{eq:LB} holds.  
Note that the lemma is clear if $p\ge K$.

Suppose $p<K$. Then $m'_p>1$.
Let $q$ and $r$ be the smallest integers
satisfying $m'_p>m'_{q+1}$ and $m'_p-1>m'_r$.
Then $m_p<m'_q$ and the inequality in \eqref{eq:LB} holds for 
$k=p,\dots,r-1$ because $m_k\le m_p\le m'_{r-1}$.
\[
\begin{matrix}
 m_1',&\ldots,&m'_{p-1},&m'_p,&\ldots,&m'_q,&m'_{q+1},&\ldots,&m'_{r-1},&m'_r\\[-7pt]
 \text{\rotatebox{270}{$=$}}&&\text{\rotatebox{270}{$=$}}&
 \text{\rotatebox{270}{$>$}}&&\text{\rotatebox{270}{$>$}}&
 \text{\rotatebox{270}{$\ge$}}&&\text{\rotatebox{270}{$\ge$}}&\\
 m_1,&\ldots,&m_{p-1},&m_p,&\ldots,&m_q,&m_{q+1},&\ldots,&m_{r-1},&m_r\\
\end{matrix}
\]
Here $p\le q<r\le K+1$, $m'_p=\cdots=m'_q>m'_{q+1}=\cdots=m'_{r-1}$
and $m'_{r-1}>m'_r$. Put
\[
  m''_j = m_j'-\delta_{j,q}+\delta_{j,r}.
\]
Then $\mathbf m''$ is monotone, $\sum (m''_j)^2 < (\sum m'_j)^2$ and
$
  m_1+\cdots+m_j\le m''_1+\cdots+m''_j\quad(j=1,2,\ldots).
$
Thus we have the lemma by the induction on the lexicographic order
of the triplet $(K-p,m_p',q)$ for a fixed $\mathbf m$.
\end{proof}
\begin{prp}\label{prp:conj}
Let $A(t)$ be a real analytic map of $(-1,1)$ to $M(n,\mathbb C)$ such
that $\dim Z_{\mathfrak g}\bigl(A(t)\bigr)$ doesn't depend on $t$.
Then there exist a partition $\mathbf m=(m_1,\dots,m_N)$ of $n$ and a 
continuous function $\lambda(t)=(\lambda_1(t),\dots,\lambda_N(t))$ of $(-1,1)$
satisfying
\begin{equation}
 A(t)\sim L\bigl(\mathbf m;\lambda(t)\bigr).
\end{equation}
\end{prp}
\begin{proof}
We find $c_j\in(-1,1)$, monotone partitions 
$\mathbf m^{(j)}\in\mathcal P^{(n)}_1$ and
real analytic functions $\lambda^{(j)}(t)=(\lambda^{(j)}_1(t),\ldots)$
on $I_j:=(c_j,c_{j+1})$ such that
\begin{gather*}
  c_{j-1}<c_j<c_{j+1},\ \lim_{\pm j\to\infty}c_j=\pm1, \ 
  A(t)\sim L\bigl(\mathbf m^{(j)};\lambda^{(j)}(t)\bigr)
  \quad(\forall t\in I_j).
\end{gather*}
Lemma~\ref{lmm:conj} assures that we may assume 
$\lambda^{(j)}(t)$ is continuous on the closure $\bar I_j$ of $I_j$ 
and $A(t)\sim L\bigl(\mathbf m^{(j)};\lambda^{(j)}(t)\bigr)$ for
$t\in\bar I_j$.  Hence $\mathbf m^{(j)}$ doesn't depend on
$j$, which we denoted by $\mathbf m$.
We can inductively define permutations 
$\sigma_{\pm j}$ of the indices $\{1,\dots,N\}$ for $j=1,2,\ldots$
so that $\sigma_0=id$, $m_{\sigma_{\pm j}(p)}=m_p$ for $p=1,\dots,N$
and moreover that 
$\bigl(\lambda^{(\nu)}_{\sigma_{\nu}(1)}(t),\dots,\lambda^{(\nu)}_{\sigma_{\nu}(N)}(t)\bigr)$ 
for $-j\le\nu\le j$ define a continuous function on $(c_{-j}, c_{j+1})$.
\end{proof}
\begin{rem} \textrm{i) } 
Suppose that $\dim Z_{M(n,\mathbb C)}\bigl(A(t)\bigr)$ is constant
for a continuous map $A(t)$ of $(-1,1)$ to $M(n,\mathbb C)$. 
For $c\in(-1,1)$ we can find $t_j\in(-1,1)$ and $\mathbf m\in\mathcal P^{(1)}$ 
such that $\lim_{j\to\infty}t_j=c$ and $\spt A(t_j)=\mathbf m$. 
The proof of Lemma~\ref{lmm:conj} shows $\spt A(c)=\mathbf m$.
Hence
\begin{equation}
 \spt A(t) \text{ doesn't depend on $t$} \ 
\Leftrightarrow \ \dim Z_{M(n,\mathbb C)}(A) \text{ doesn't depend on $t$}.
\end{equation}

\textrm{ii) }
It is easy to show that Proposition~\ref{prp:conj} is valid even if 
we replace ``real analytic" by ``continuous" 
but it is not true if we replace ``real analytic" and ``$(-1,1)$" 
by ``holomorphic" and ``$\{t\in\mathbb C\,;\,|t|<1\}$", respectively. 
The matrix 
$A(t)=\left(\begin{smallmatrix}0 & 1\\t & 0\end{smallmatrix}\right)$ 
is a counter example.
\end{rem}
\section{Deligne-Simpson problem}
For simplicity we put $\mathfrak g=M(n,\mathbb C)$ and $G=GL(n,\mathbb C)$
only in this section.

Let $\mathbf A=(A_0,\dots,A_k)\in\mathfrak g^{k+1}$.
Put
\begin{align}
 M(n,\mathbb C)^{k+1}_0&:=\{(C_0,\dots,C_k)\in\mathfrak g^{k+1}\,;\,
 C_0+\cdots+C_k=0\},\\
 Z_{\mathfrak g}(\mathbf A)&:=\{X\in\mathfrak g\,;\,[A_j,X]=0\ \ 
(j=0,\dots,k)\}.
\end{align}
A tuple of matrices  $\mathbf A\in\mathfrak g^{k+1}$
is called {\sl irreducible} if any subspace $V\subset\mathbb C^n$ 
satisfying $A_jV\subset V$ for $j=0,\dots,k$ is $\{0\}$ or $\mathbb C^n$.

Suppose $\trace A_0+\dots+\trace A_k=0$.
The {\sl additive Deligne-Simpson problem} presented by Kostov \cite{Ko}
is to determine the condition to $\mathbf A$ for the existence of
an irreducible tuple $\mathbf B=(B_0,\dots,B_k)\in M(n,\mathbb C)^{k+1}_0$ 
satisfying $A_j\sim B_j$ for $j=0,\dots,k$. 
The condition is concretely given by Crawley-Boevey \cite{CB} 
(cf.~Theorem~\ref{thm:CB} and \cite{Ko4}).

Suppose $\mathbf A\in M(n,\mathbb C)^{k+1}_0$. 
Then $\mathbf A$ is called {\sl rigid} if $\mathbf A\sim \mathbf B$
for any element $\mathbf B=(B_0,\dots,B_k)\in M(n,\mathbb C)^{k+1}_0$
satisfying $B_j\sim A_j$ for $j=0,\dots,k$.
Here we denote $\mathbf A\sim\mathbf B$ if there exists $g\in G$ with
$(B_0,\dots,B_k)=(gA_0g^{-1},\dots,gA_kg^{-1})$.

\begin{rem}
Note that
the local monodromy at $\infty$ of the Fuchsian system
\begin{equation}\label{eq:Fuchs}
  \frac{du}{dz} = \sum_{j=1}^k\frac{A_j}{z-z_j}u
\end{equation}
on a Riemann sphere corresponds to $A_0$
with $\mathbf A=(A_0,\dots,A_k)\in M(n,\mathbb C)^{k+1}_0$.
Then the quotient $M(n,\mathbb C)^{k+1}_0\!/\!\!\sim$ classifies the Fuchsian
systems.
\end{rem}

Under the identification of $\mathfrak g$ with its dual space
by the symmetric bilinear form $\langle X,Y\rangle=\trace XY$
for $(X,Y)\in\mathfrak g^2$, 
the dual map of $\ad_A:X\mapsto[A,X]$ of $\mathfrak g$ equals 
$-\ad_A$ and therefore $\ad_A(\mathfrak g)$
is the orthogonal compliment of $\ker\ad_A$ under the bilinear form:
\begin{equation}\label{eq:ad}
 \ad_A(\mathfrak g):=
 \{[A,X]\,;\,X\in\mathfrak g\}=\{X\in\mathfrak g\,;\,\trace XY=0
\quad(\forall Y\in Z_{\mathfrak g}(A))\}.
\end{equation}

For $\mathbf A=(A_0,\dots,A_k)\in \mathfrak g^{k+1}$
we put
\begin{equation*}
 \begin{matrix}
 \pi_{\mathbf A}\,:&G^{k+1}&\to&\mathfrak g\\
         &\rotatebox{90}{$\in$}&&\rotatebox{90}{$\in$}\\
         &(g_0,\dots,g_k)&\mapsto&\sum_{j=0}^k g_jA_jg_j^{-1}
 \end{matrix}
\end{equation*}
The image of $\pi_{\mathbf A}$ is a homogeneous space $G^{k+1}/H$
of $G^{k+1}$ with
\[
  H:=\{(g_0,\dots,g_k)\in G^{k+1}\,;\,
  \sum_{j=0}^k g_jA_jg_j^{-1}=\sum_{j=0}^kA_j
 \}
\]
and the tangent space of the image at
$A_0+\dots+A_k$ is isomorphic to
\begin{equation*}
 \sum_{j=0}^k \ad_{A_j}(\mathfrak g)
 =\bigl\{X\in\mathfrak g\,;\,\trace XY = 0\quad\bigl(\forall Y\in
 Z_{\mathfrak g}(\mathbf A)=\bigcap_{j=0}^{k}Z_{\mathfrak g}(A_j)\bigr)\bigr\}.
\end{equation*}
Hence the dimension of the manifold $G^{k+1}/H$ 
equals $n^2 - \dim  Z_{\mathfrak g}(\mathbf A)$ and therefore the dimension
of $H$ equals $kn^2+\dim  Z_{\mathfrak g}(\mathbf A)$.
Since the manifold
\begin{equation}
 \widetilde O_{\mathbf A}:=\{(C_0,\dots,C_k)\in\mathfrak g^{k+1}\,;\,
 C_j\sim A_j\text{ and }\sum_{j=0}^k C_j=\sum_{j=0}^k A_j\}
\end{equation}
is naturally isomorphic to 
$H/Z_{G}(A_0)\times\cdots\times Z_{G}(A_k)$ with
$Z_G(A_j):=\{g\in G\,;\,gA_jg^{-1}=A_j\}$,
the dimension of $\widetilde O_{\mathbf A}$ equals 
$kn^2+\dim  Z_{\mathfrak g}(\mathbf A)-\sum_{j=0}^k
\dim  Z_{\mathfrak g}(\mathbf A_j)$.

Note that the dimension of the manifold
\begin{equation}
  O_{\mathbf A}:=\bigcup_{g\in G}(gA_0g^{-1},\dots,gA_kg^{-1})
 \subset\mathfrak g^{k+1}
\end{equation}
equals $n^2 - \dim  Z_{\mathfrak g}(\mathbf A)$.

Suppose $\mathbf A\in M(n,\mathbb C)^{k+1}_0$.
Then $\widetilde O_{\mathbf A}\supset O_{\mathbf A}$
and we have the followings.
\begin{prp}
$\dim\widetilde O_{\mathbf A}-\dim O_{\mathbf A}
 = (k-1)n^2-\displaystyle\sum_{j=0}^k\dim  Z_{\mathfrak g}(A_j)
  +2\dim  Z_{\mathfrak g}(\mathbf A)$.
\end{prp}
\begin{dfn}\label{defn:pidx}
The index of rigidity $\idx \mathbf A$ of $\mathbf A$ 
is introduced by \cite{Kz}:
\begin{align*}
 \idx\mathbf A &:=\sum_{j=0}^k\dim Z_{\mathfrak g}(A_j) - (k-1)n^2
 =2n^2 - \sum_{j=0}^k \dim\{gA_jg^{-1}\,;\,g\in G\},\\
 \Pidx\mathbf A
 &:=\dim Z_{\mathfrak g}(\mathbf A)
  + \tfrac12(k-1)n^2- \tfrac12\sum_{j=0}^k \dim Z_{\mathfrak g}(A_j)
  = \dim Z_{\mathfrak g}(\mathbf A) - \tfrac12\idx\mathbf m.
\end{align*}
Note that $\Pidx\mathbf A\ge0$ and $\dim\{gA_jg^{-1}\,;\,g\in G\}$ are even.
\end{dfn}
\begin{crl}
$\dim\widetilde O_{\mathbf A}-\dim O_{\mathbf A}$ and\/ $\idx\mathbf A$ are 
even and $\idx\mathbf A\le2\dim Z_{\mathfrak g}(\mathbf A)$.
\end{crl}

Note that if $\mathbf A$ is irreducible, 
$\dim Z_{\mathfrak g}(\mathbf A)=1$.

The following result by Katz is fundamental.
\begin{thm}[\cite{Kz}]
Suppose $\mathbf A\in M(n,\mathbb C)^{k+1}_0$ is irreducible.
Then $\idx\mathbf A=2$ if and only if 
$\mathbf A$ is rigid, namely, $\widetilde O_{\mathbf A}=O_{\mathbf A}$.
\end{thm}
\section{Middle convolutions}
We will review the additive middle convolutions
in the way interpreted by Dettweiler and Reiter \cite{DR, DR2}.
\begin{dfn}[\cite{DR}]
Fix $\mathbf A=(A_0,\dots,A_k)\in M(n,\mathbb C)^{k+1}_0$.
The {\sl addition} $M_{\mu'}(\mathbf A)\in\mathfrak g^{k+1}$ of 
$\mathbf A$ with respect to $\mu'=(\mu'_1,\dots,\mu'_k)\in\mathbb C^k$ is
$(A_0-\mu'_1-\dots-\mu'_k,A_1+\mu'_1,\dots,A_k+\mu'_k)$.
The {\sl convolution} $(G_0.\dots,G_k)\in M(kn,\mathbb C)^{k+1}_0$ 
of $\mathbf A$ with respect to $\lambda\in\mathbb C$ 
is defined by 
\begin{align}
  G_j &= \Bigl(\delta_{p,j}(A_q+\delta_{p,q}\lambda)\Bigr)_{\substack{1\le p\le k\\ 1\le q\le k}}\qquad(j=1,\dots,k)\\
  &=
  \bordermatrix{
         &    &     &        & \underset{\smallsmile}{j} \cr
      & \cr
    j\,{\text{\tiny$)$}} &A_1 & A_2 & \cdots & A_j+\lambda & A_{j+1} & \cdots & A_k\cr
      & \cr
  },\notag\\
  G_0 &= -(G_1+\dots+G_k).
\end{align}
Put $\mathcal K=\Bigl\{
    \left(\begin{smallmatrix}v_1\\ \vdots\\v_k\end{smallmatrix}\right)
    \,;\,v_j\in\ker A_j
\quad(j=1,\dots,k)\Bigr\}$ and $\mathcal L=\ker G_0$.
Then $\mathcal K$ and $\mathcal L$ are $G_j$-invariant 
subspaces of $\mathbb C^{kn}$
and we define $\bar G_j:=G_j|_{\mathbb C^{kn}/(\mathcal K+\mathcal L)}
\in\End(\mathbb C^{n'})\simeq M(n',\mathbb C)$ with 
$n'=kn-\dim(\mathcal K+\mathcal L)$.  The {\sl middle convolution}
$mc_\lambda(\mathbf A)\in M(n',\mathbb C)^{k+1}_0$ of $\mathbf A$ with 
respect to $\lambda$ is defined by 
$mc_\lambda(\mathbf A):=(\bar G_0,\dots,\bar G_k)$.
Note that $\mathcal K\cap\mathcal L=\{0\}$ if $\lambda\ne0$.
\end{dfn}

The conjugacy classes of $\bar G_j$ in the above definition are given
by \cite{DR2}, which is simply described using the normal form in 
\S\ref{S:M} (cf.~Proposition~\ref{prp:conj}):
\begin{thm}[\cite{DR, DR2}]\label{thm:mc}
Fix $\mathbf A=(A_0,A_1,\dots,A_k)\in M(n,\mathbb C)^{k+1}_0$ and
$
   \mu =(\mu_0,\dots,\mu_k)\in\mathbb C^{k+1}
$
and put
\begin{equation}
 \begin{split}
 mc_{\mu}&:=M_{-\mu'}\circ mc_{|\mu|}\circ M_{-\mu'},\\
 \mu'&:=(\mu_1,\dots,\mu_k),\quad
 |\mu|:=\mu_0+\mu_1+\cdots+\mu_k.
 \end{split}
\end{equation}
Assume the following conditions (which are satisfied if $n>1$ and $\mathbf A$
is irreducible):
\begin{align}
  \bigcap_{\substack{1\le j\le k\\ j\ne i}}\ker(A_j-\mu_j)
  \cap \ker(A_0-\tau)&=\{0\}&(i=1,\dots,k,\ \forall\tau\in\mathbb C)
  \label{eq:star}\\
  \sum_{\substack{1\le j\le k\\ j\ne i}}\IM(A_j-\mu_j)
  +\IM(A_0-\tau)&=\mathbb C^n&(i=1,\dots,k,\ \forall\tau\in\mathbb C)
  \label{eq:starstar}
\end{align}
Then $\mathbf A':=mc_{\mu}(\mathbf A)$ satisfies \eqref{eq:star} and
\eqref{eq:starstar} with replacing $-\mu_j$ by $+\mu_j$ and
\begin{equation}
 \idx \mathbf A'=\idx\mathbf A.
\end{equation}
If $\mathbf A$ is irreducible, so is $\mathbf A'$.  
If $\mu=0$, then $\mathbf A'\sim\mathbf A$.
If $\mathbf A\sim\mathbf B$, then $mc_\mu(\mathbf A)\sim mc_\mu(\mathbf B)$.
Moreover for any $\tau_0\in\mathbb C$ we have
\begin{align}
 mc_{(-\tau_0,\,-\mu')}\circ mc_{(\mu_0,\,\mu')}(\mathbf A)&\sim
 M_{2\mu'}\circ mc_{(2\mu_0-\tau_0-|\mu|,\,\mu')}(\mathbf A),\\
 mc_{-\mu}\circ mc_\mu(\mathbf A)&\sim\mathbf A.
\end{align}
Choose 
$\mathbf m\in\mathcal P^{(n)}_{k+1}$ and 
$\lambda_{j,\nu}\in\mathbb C$
so that
\begin{equation}\label{eq:DS}
 A_j\sim L\bigl(\mathbf m_j;\lambda_j)
 \text{ with }\mathbf m_j:=(m_{j,1},\dots,m_{j,n_j})\text{ and }
 \lambda_j:=(\lambda_{j,1},\dots,\lambda_{j,n_j}).
\end{equation}
Denoting
$I_j:=\{\nu\,;\,\lambda_{j,\nu}=\mu_j\}$ and putting
\begin{align}
 \ell_j=&\begin{cases}
         \min\bigl\{p\in I_j\,;\,m_p=\max\{m_\nu\,;\,\nu\in I_j\}\bigr\}
         &(I_j\ne\emptyset)\\
         n_j+1&(I_j=\emptyset)
       \end{cases},\\
 d_\ell(\mathbf m)&:=
 m_{0,\ell_0}+m_{1,\ell_1}+\cdots+m_{k,\ell_k}-(k-1)n,\label{eq:d_ell}\\
 m_{j,\nu}'&:=
 m_{j,\nu}-\delta_{\ell_j,\nu}\cdot d_\ell(\mathbf m),\label{eq:p(m)}\\
 \lambda_{j,\nu}'&:=
  \begin{cases}
   \lambda_{j,\nu}+|\mu|-2\mu_j &(\nu\ne\ell_j)\\
   -\mu_j                       &(\nu=\ell_j)
  \end{cases},
\end{align}
we have $A_j'\sim L(\mathbf m_j';\lambda'_j)$ \ $(j=0,\dots,k)$
if\/ $|\mu|\ne0$.
\end{thm}
\begin{exa} 
Suppose $\lambda_i$, $\mu_j$ and $\tau_\ell$ are generic. 
Starting from $\mathbf A=(-\lambda_1-\lambda_2,\lambda_1,\lambda_2)\in M(1,\mathbb C)^3_0$,
we have the following list of eigenvalues of the matrices 
under the application of middle convolutions to $\mathbf A$
(cf.~hypergeometric family in Example~\ref{ex:simpson}):

\qquad
$1,1,1\ (H_1)\ \longleftrightarrow \ 11,11,11
\ (H_2:{}_2F_1)\longleftrightarrow \ 111,111,12\ (H_3:{}_3F_2)$
\if\RIMS
\begin{align*}
 &\begin{Bmatrix}-\lambda_1-\lambda_2&\ \ \ \lambda_1
   &\ \ \ \lambda_2\end{Bmatrix}
 \xrightarrow
  {mc_{\mu_0,\mu_1,\mu_2}}\\
 &\begin{Bmatrix}
    -\lambda_1-\lambda_2-\mu_0+\mu_1+\mu_2&\ \ \ \lambda_1+\mu_0-\mu_1+\mu_2 &\ \ \ \lambda_2+\mu_0+\mu_1-\mu_2\\
    -\mu_0&\ \ \ -\mu_1 &\ \ \ -\mu_2
 \end{Bmatrix}
  \xrightarrow
  {mc_{\tau_0,\tau_1,-\mu_2}}\\
 &\begin{Bmatrix}
    -\lambda_1-\lambda_2-\mu_0+\mu_1-\tau_0+\tau_1&\ \ \ 
    \lambda_1+\mu_0-\mu_1+\tau_0-\tau_1
     &\ \ \ \lambda_2+\mu_0+\mu_1+\tau_0+\tau_1 
  \\
   -\mu_0-\tau_0+\tau_1-\mu_2&\ \ \ -\mu_1+\tau_0-\tau_1-\mu_2&\ \ \ \mu_2\\
   -\tau_0&\ \ \ -\tau_1 &\ \ \ \mu_2
 \end{Bmatrix}
\end{align*}
\else
\begin{align*}
 &\begin{Bmatrix}-\lambda_1-\lambda_2&\lambda_1
   &\lambda_2\end{Bmatrix}
 \xrightarrow
  {mc_{\mu_0,\mu_1,\mu_2}}\\
 &\begin{Bmatrix}
    -\lambda_1-\lambda_2-\mu_0+\mu_1+\mu_2&\lambda_1+\mu_0-\mu_1+\mu_2 &\ \lambda_2+\mu_0+\mu_1-\mu_2\\
    -\mu_0&-\mu_1 &-\mu_2
 \end{Bmatrix}
  \xrightarrow
  {mc_{\tau_0,\tau_1,-\mu_2}}\\
 &\begin{Bmatrix}
    -\lambda_1-\lambda_2-\mu_0+\mu_1-\tau_0+\tau_1&
    \lambda_1+\mu_0-\mu_1+\tau_0-\tau_1
     &\lambda_2+\mu_0+\mu_1+\tau_0+\tau_1 
  \\
   -\mu_0-\tau_0+\tau_1-\mu_2&-\mu_1+\tau_0-\tau_1-\mu_2&\mu_2\\
   -\tau_0&\ \ \ -\tau_1 &\mu_2
 \end{Bmatrix}
\end{align*}
\fi
Here 
the eigenvalues are vertically written.
Note that the matrices are semisimple if the parameters are generic.
Denoting $\mathbf A'=(A'_0, A'_1, A'_2)=mc_{\mu_0,\mu_1,\mu_2}(\mathbf A)$ and $\mathbf A''=(A''_0,A''_1,A''_2)=mc_{\tau_0,\tau_1,-\mu_2}(\mathbf A')$, we have
\begin{align}
 \begin{split}
 A'_0&\sim L(1,1;-\lambda_1-\lambda_2-\mu_0+\mu_1+\mu_2,-\mu_0),
\\ 
 A'_j&\sim L(1,1;\lambda_j+\mu_0+\mu_1+\mu_2-2\mu_j,-\mu_j)\quad(j=1,\,2),
 \end{split} \label{eq:H21}\\ 
 A''_2&\sim L(1,2;\lambda_2+\mu_0+\mu_1+\tau_0+\tau_1,\mu_2)\text{, etc}.
\end{align}
Then Theorem~\ref{thm:mc} implies that the irreducible rigid tuple
$\mathbf A=(A'_0,A'_1,A'_2)\in M(2,\mathbb C)^3_0$ satisfying
\eqref{eq:H21} exists if and only if
$\lambda_1\ne\mu_1$, $\lambda_2\ne\mu_2$, $\lambda_1+\lambda_2+\mu_0\ne0$
and $\mu_0+\mu_1+\mu_2\ne 0$, which corresponds to 
$\mathcal K=\mathcal L=\{0\}$ and $|\mu|\ne0$.
Moreover all the irreducible rigid tuples $\mathbf A\in M(2,\mathbb C)^3_0$ are
obtained in this way.
\end{exa}
\begin{dfn}\label{dfn:rigid}
{\rm i) }
Under the notation in Theorem~\ref{thm:mc} the tuple of
partitions $\mathbf m\in\mathcal P^{(n)}_{k+1}$ is called 
the \textsl{spectral type} of $\mathbf A$
and denoted by $\spt\mathbf A$.

{\rm ii) }
Let $\mathbf m\in\mathcal P^{(n)}_{k+1}$
and  $\lambda_{j,\nu}$ be generic complex numbers satisfying
\begin{equation}\label{eq:sum0}
  \sum_{j=0}^k\sum_{\nu=1}^{n_j}m_{j,\nu}\lambda_{j,\nu}=0.
\end{equation}
Then $\mathbf m$ is \textsl{realizable} if there exists a tuple 
$\mathbf A\in M(n,\mathbb C)^{k+1}_0$ satisfying \eqref{eq:DS}.
Moreover $\mathbf m$ is \textsl{irreducibly realizable}
if there exists an irreducible tuple $\mathbf A\in M(n,\mathbb C)^{k+1}_0$
satisfying \eqref{eq:DS}.
An irreducibly realizable tuple $\mathbf m$ is \textsl{rigid} if 
$\idx \mathbf m:=\idx(\mathbf m,\mathbf m)=2$, namely, the corresponding
irreducible tuple $\mathbf A$ is rigid.

For 
$\ell=(\ell_0,\dots,\ell_k)\in\mathbb Z^{k+1}_{\ge 1}$ 
we define 
$\p_{\mathbf\ell}(\mathbf m)=\mathbf m'$ by
\eqref{eq:d_ell} and \eqref{eq:p(m)}
and denote the unique monotone element in 
$S_\infty'\mathbf m$ by $s(\mathbf m)$.
Moreover we define
\begin{align}
 \p(\mathbf m)&:=\p_{(1,1,\ldots)}(\mathbf m)=\p_{\bf 1}(\mathbf m),
  \label{eq:p}\\
 \p_{max}(\mathbf m)&:=\p_\ell(\mathbf m)\text{ \ with \ }
  \ell_j=\min\bigl\{\nu\,;\,m_{j,\nu}=\max\{m_{j,1},m_{j,2},\ldots\}\bigr\}
  \label{eq:pmax}
\end{align}
and $\mathbf m$ is \textit{basic} if $\mathbf m$ is indivisible and
$\sum_{j=0}^k\max\{m_{j,1},m_{j,2},\ldots\}\le (k-1)\ord\mathbf m$
which means $\ord\p_{max}(\mathbf m)\ge\ord\mathbf m$.
Under the notation \eqref{eq:pmax} and \eqref{eq:DS} we put
\begin{equation}
 mc_{max}(\mathbf A):=mc_{\lambda_{\ell_0},\lambda_{\ell_1},\ldots}(\mathbf A).
\end{equation}
\end{dfn}
\begin{rem} 
%
{\rm i)\ }
Suppose $\mathbf m\in\mathcal P_{k+1}$ is irreducibly realizable.
Then $mc_\ell(\mathbf m)\in\mathcal P_{k+1}$ if
$\#\{(j,\nu)\,;\,m_{j,\nu}>0\text{ and }\nu\ne\ell_j\}>1$.
Moreover if $\mathbf A$ is a generic element of $M(n,\mathbb C)^{k+1}_0$
satisfying $\spt\mathbf A=\mathbf m$ and moreover
$\mu=(\mu_0,\dots,\mu_k)\in\mathbb C^{k+1}$ is generic under the condition that
$\mu_j=\lambda_{j,\ell_j}$ for any $\ell_j$ satisfying $m_{j,\ell_j}>0$,
then $mc_\mu(\mathbf A)$ is a generic element of $M(n,\mathbb C)^{k+1}_0$
with the spectral type $\p_\ell(\mathbf m)$.

\smallskip
{\rm ii)\ }
Let $\mathbf A\in M(n,\mathbb C)^{k+1}_0$ with a spectral type $\mathbf m$.
Let $\ell=(\ell_0,\ell_1,\dots)$ with $\ell_j\in\mathbb Z_{> 0}$ and
$\ell_\nu=1$ for $\nu>k$.
Define $\mathbf 1_\ell=(m'_{j,\nu})\in\mathcal P^{(1)}$ by
$m'_{j,\nu}=\delta_{\ell_j,\nu}$.  Then
\begin{align}
 \idx \mathbf A&=\idx\mathbf m:=\idx(\mathbf m,\mathbf m),\\
  d_\ell(\mathbf m)&=\idx(\mathbf m,\mathbf 1_\ell).
\end{align}
\end{rem}
\begin{thm}\label{thm:IrReal}
{\rm i)\ \ (\cite{Kz}, \cite{DR})\ }
Let $\mathbf A\in M(n,\mathbb C)^{k+1}_0$ and put $\mathbf m=\spt\mathbf A$.
Then $\mathbf A$ is irreducible and rigid if 
and only if $n=1$ or $mc_{max}(\mathbf A)$ is irreducible and rigid
and $\ord\p_{max}(\mathbf m)<n$.
Hence if $\mathbf A$ is irreducible and rigid, $\mathbf A$ is
constructed from an element of $M(1,\mathbb C)^{k+1}_0$ by a finite
iteration of suitable middle convolutions $mc_\mu$ in Theorem~\ref{thm:mc}.

\smallskip
{\rm ii)\ \ (\cite{Ko4}, \cite{CB})\ }
An indivisible tuple $\mathbf m\in\mathcal P$ is irreducibly realizable 
if and only if one of the following three conditions holds.
\begin{align}
&\ord\mathbf m=1\\
&\mathbf m\text{ is \textsl{basic}, namely, $\mathbf m$ is indivisible and }
\ord \p_{max}(\mathbf m)\ge \ord\mathbf m\label{eq:basic}\\
&\p_{max}(\mathbf m)\in\mathcal P\text{ is well-defined and irreducibly 
realizable}.
\end{align}
Note that $\p_\ell(\mathbf m)\in\mathcal P$ is well-defined if and only if
$m_{j,\ell_j}\ge d_\ell(\mathbf m)$ for $j=0,1,\ldots$.

\smallskip
{\rm iii)}\  {\rm(Theorem~\ref{thm:GDS} in \S\ref{sec:apd})}
Suppose a tuple $\mathbf m\in \mathcal P$ is not indivisible.
Put $\mathbf m=d\overline{\mathbf m}$ with an integer $d>1$ 
and an indivisible tuple $\overline{\mathbf m}\in\mathcal P$.
Then $\mathbf m$ is irreducibly realizable if and only if
$\overline{\mathbf m}$ is irreducibly realizable and $\idx\mathbf m<0$.
\end{thm}
\begin{exa}
Successive applications of $s\circ\p$ to monotone elements
of $\mathcal P$:
\iffalse
\begin{align*}
&\underline411,\underline411,\underline42,\underline33
\overset{15-2\cdot6=3}\longrightarrow
111,111,21,3=\underline111,\underline111,\underline21
\overset{4-3=1}\longrightarrow{\underline1}1,{\underline1}1,{\underline1}1
\overset{3-2=1}\longrightarrow1,1,1 &(\text{rigid})
\\
&\underline211,\underline211,\underline1111\overset{5-4=1}\longrightarrow
\underline111,\underline111,\underline111\overset{3-3=0}\longrightarrow 111,111,111& (\text{realizable, not rigid})
\\
&\underline211,\underline211,\underline211,\underline31
\overset{9-8=1}\longrightarrow
\underline111,\underline111,\underline111,\underline21
\overset{5-6=-1}\longrightarrow& (\text{realizable, not rigid})
\\
&{\underline2}2,{\underline2}2,{\underline1}111\overset{5-4=1}\longrightarrow
{\underline2}1,{\underline2}1,{\underline1}11
\overset{5-3=2}\longrightarrow\times&(\text{not realizable})
\end{align*}
\else

\noindent
$\underline411,\underline411,\underline42,\underline33
\overset{15-2\cdot6=3}\longrightarrow
111,111,21,3=\underline111,\underline111,\underline21
\overset{4-3=1}\longrightarrow{\underline1}1,{\underline1}1,{\underline1}1
\overset{3-2=1}\longrightarrow1,1,1$\hfill(rigid)

\noindent
$\underline211,\underline211,\underline1111\overset{5-4=1}\longrightarrow
\underline111,\underline111,\underline111\overset{3-3=0}\longrightarrow 111,111,111$\hfill(realizable, not rigid)

\noindent
\if\RIMS
$\underline211,\underline211,\underline211,\underline31
\overset{9-8=1}\longrightarrow
\underline111,\underline111,\underline111,\underline21
\overset{5-6=-1}\longrightarrow211,211,211,31$\hfill(realizable, not rigid)
\else
$\underline211,\underline211,\underline211,\underline31
\!\overset{9-8=1}\longrightarrow\!
\underline111,\underline111,\underline111,\underline21
\!\overset{5-6=-1}\longrightarrow\!211,211,211,31$\ (realizable, not rigid)
\fi

\noindent
${\underline2}2,{\underline2}2,{\underline1}111\overset{5-4=1}\longrightarrow
{\underline2}1,{\underline2}1,{\underline1}11
\overset{5-3=2}\longrightarrow\times$\hfill(not realizable)\\[2pt]
The numbers on the above arrows are $d_{(1,1,\dots)}(\mathbf m)
=m_{0,1}+\cdots+m_{k,1}-(k-1)\cdot\ord\mathbf m$.
\fi
\end{exa}

\section{Rigid tuples}
Let $\mathcal R_k^{(n)}$ denote the totality of rigid tuples in 
$\mathcal P_k^{(n)}$ (cf.~Definition~\ref{dfn:rigid}).  
Put $\mathcal R_k=\bigcup_{n=1}^\infty \mathcal R_k^{(n)}$, 
$\mathcal R^{(n)}=\bigcup_{k=1}^\infty \mathcal R_k^{(n)}$
and $\mathcal R=\bigcup_{n=1}^\infty \mathcal R_k$.
We identify elements of $\mathcal R$ if they are in the 
same $S_\infty$-orbit (cf.~Definition~\ref{dfn:Sinfty}) and then 
$\bar{\mathcal R}$ denotes the set of elements of $\mathcal R$ 
under this identification.
Similarly we denote ${\bar{\mathcal R}}_k$ and ${\bar{\mathcal R}}^{(n)}$ for 
$\mathcal R_k$ and $\mathcal R^{(n)}$, respectively,  with this
identification.
\begin{exa}\label{ex:simpson} {\rm i)\  }
The list of $\mathbf m\in\bar{\mathcal R}^{(n)}$ with 
$\mathbf m_0=1^n$ is given by Simpson \cite{Si}:
\begin{align*}
 &1^n,1^n,n-11\text{\ ($H_n$:\,hypergeometric family)}
 &&1^{2m},mm,mm-11\text{\ ($EO_{2m}$:\,even family)}\\
 &1^{2m+1},m+1m,mm1\text{\ ($EO_{2m+1}$:\,odd family)}
 &&111111,222,42\text{\ ($X_6$:\,extra case)}
\end{align*}

{\rm ii)\ }
We show examples and the numbers of elements of 
${\bar{\mathcal R}}^{(n)}$.

\medskip
\centerline{{\bf Table} \ ${\bar{\mathcal R}}^{(n)}$ ($2\le n\le 7$)}
\if\RIMS
\vspace{-10pt}
\else
\vspace{-4pt}
\fi
{\small\begin{verbatim}
2:11,11,11                3:111,111,21               3:21,21,21,21      
4:1111,1111,31            4:1111,211,22              4:211,211,211
4:211,22,31,31            4:22,22,22,31              4:31,31,31,31,31
5:11111,11111,41          5:11111,221,32             5:2111,2111,32
5:2111,221,311            5:221,221,221              5:221,221,41,41
5:221,32,32,41            5:311,311,32,41            5:32,32,32,32
5:32,32,41,41,41          5:41,41,41,41,41,41        6:111111,111111,51
6:111111,222,42           6:111111,321,33            6:21111,2211,42
6:21111,222,33            6:21111,222,411            6:21111,3111,33
6:2211,2211,33            6:2211,2211,411            6:2211,222,51,51
6:2211,321,321            6:2211,33,42,51            6:222,222,321
6:222,3111,321            6:222,33,33,51             6:222,33,411,51
6:3111,3111,321           6:3111,33,411,51           6:321,321,42,51
6:321,33,51,51,51         6:321,42,42,42             6:33,33,33,42
6:33,33,411,42            6:33,411,411,42            6:33,42,42,51,51
6:411,411,411,42          6:411,42,42,51,51          6:51,51,51,51,51,51,51
7:1111111,1111111,61      7:1111111,331,43           7:211111,2221,52
7:211111,322,43           7:22111,22111,52           7:22111,2221,511
7:22111,3211,43           7:22111,331,421            7:2221,2221,43
7:2221,2221,61,61         7:2221,31111,43            7:2221,322,421
7:2221,331,331            7:2221,331,4111            7:2221,43,43,61
7:31111,31111,43          7:31111,322,421            7:31111,331,4111
7:3211,3211,421           7:3211,322,331             7:3211,322,4111
7:3211,331,52,61          7:322,322,322              7:322,322,52,61
7:322,331,511,61          7:322,421,43,61            7:322,43,52,52
7:331,331,43,61           7:331,331,61,61,61         7:331,43,511,52
7:4111,4111,43,61         7:4111,43,511,52           7:421,421,421,61
7:421,421,52,52           7:421,43,43,52             7:421,43,511,511
7:421,43,52,61,61         7:43,43,43,43              7:43,43,43,61,61
7:43,43,61,61,61,61       7:43,52,52,52,61           7:511,511,52,52,61
7:52,52,52,61,61,61       7:61,61,61,61,61,61,61,61
\end{verbatim}%
}

\qquad\qquad\qquad
$\mathcal R_{k}^{(n)}$: rigid $k$-tuples of partitions with order 
$n\phantom{\displaystyle\frac{A_B}{C_D}}$\\
\smallskip
\nopagebreak
\centerline{
\begin{tabular}{|r|r|r||r|r|r||r|r|r|}\hline
{\small ord} & ${\#\bar{\mathcal R}}_3^{(n)}\!\!$
 & ${\#\bar{\mathcal R}}^{(n)}\!\!$ &
{\small ord} & ${\#\bar{\mathcal R}}_3^{(n)}$ & $\#\bar{\mathcal R}^{(n)}$ &
{\small ord} & ${\#\bar{\mathcal R}}_3^{(n)}$ & $\#\bar{\mathcal R}^{(n)}$ 
 \\ \hline
2 & 1 & 1 & 15&1481&2841 &28&114600&190465\\ \hline
3 & 1 & 2 & 16&2388&4644 &29&143075&230110\\ \hline
4 & 3 & 6 & 17&3276&6128 &30&190766&310804\\ \hline
5 & 5 & 11& 18&5186&9790 &31&235543&371773\\ \hline
6 & 13& 28& 19&6954&12595&32&309156&493620\\ \hline
7 & 20& 44& 20&10517&19269&33&378063&588359\\ \hline
8 & 45& 96& 21&14040&24748&34&487081&763126\\ \hline
9 & 74&157& 22&20210&36078&35&591733&903597\\ \hline
10&142&306& 23&26432&45391&36&756752&1170966\\ \hline
11&212&441& 24&37815&65814&37&907150&1365027\\ \hline
12&421&857& 25&48103&80690&38&1143180&1734857\\ \hline
13&588&1177& 26&66409&112636&39&1365511&2031018\\ \hline
14&1004&2032& 27&84644&139350&40&1704287&2554015\\ \hline
\end{tabular}}
\end{exa}
\section{A Kac-Moody root system}
We will review the relation between a Kac-Moody root system and the middle 
convolution which is clarified by [CB].

Let $\mathfrak h$ be an infinite dimensional real vector space with 
the set of basis $\Pi$, where
\begin{equation}
 \Pi=\{\alpha_0,\alpha_{j,\nu}\,;\,j=0,1,2,\ldots,\ 
 \nu=1,2,\ldots\}.
\end{equation}
Put
\begin{equation}
 Q  :=\sum_{\alpha\in\Pi}\mathbb Z\alpha\ \supset\ 
 Q_+:=\sum_{\alpha\in\Pi}\mathbb Z_{\ge0}\alpha.
\end{equation}
We define an indefinite inner product on $\mathfrak h$ by
\begin{equation}
 \begin{split}
 (\alpha|\alpha)
 &= 2\qquad\ \,(\alpha\in\Pi),\\
 (\alpha_0|\alpha_{j,\nu})
 &=-\delta_{\nu,1}\quad\!(j=0,1,\ldots,\ \nu=1,2,\ldots),\\
 (\alpha_{i,\mu}|\alpha_{j,\nu})
 &=\begin{cases}
    0 &(i\ne j\text{ \ or \ }|\mu-\nu|>1)\\
    -1&(i=j\text{ \ and \ }|\mu-\nu|=1)
  \end{cases}.
 \end{split}
\end{equation}
Let $\mathfrak g_\infty$ denote the Kac-Moody Lie
algebra associated to the Cartan matrix
\begin{align}
  A&:=\left(\frac{2(\alpha_i|\alpha_j)}{(\alpha_i|\alpha_i)}
    \right)_{i,j\in I},\\
  I&:=\{0,\,(j,\nu)\,;\,j=0,1,\ldots,\ \nu=1,2,\ldots\}.
\end{align}
We introduce linearly independent vectors $e_0$ and 
$e_{j,\nu}$ ($j=0,1,\ldots,\ \nu=1,2,\ldots$) with
\begin{equation}
 (e_0|e_0) = 2,\ 
 (e_0|e_{j,\nu}) = -\delta_{\nu,1}\text{ \ and \ }
 (e_{j,\nu}|e_{j',\nu'}) = \delta_{j,j'}\delta_{\nu,\nu'}.
\end{equation}

For a sufficiently large positive integer $k$ let ${\mathfrak h}^k$ be
a subspace of $\mathfrak h$ spanned by $\{\alpha_0,\,\alpha_{j,\nu}\,;\,
j=0,1,\dots,k,\ \nu=0,1,\ldots\}$.
Putting $e_0^k=e_0+e_{0,1}+\cdots+e_{k,1}$, we have
$(e_0^k|e_0^k) = 2+(k+1) - 2(k+1)=1-k$.
For a sufficiently large $k$ we have an orthogonal basis
$\{e_0^k,\ e_{j,\nu}\,;\,j=0,\dots,k,\ \nu=1,2,\ldots\}$
with
\begin{align}
 \begin{split}
  (e_0^k|e_0^k) &= 1-k,\quad
  (e_{j,\nu}|e_{j',\nu'}) = \delta_{j,j'}\delta_{\nu,\nu'},\\
  (e_0^k|e_{j,\nu}) &= 0
  \qquad(j=0,\dots,k,\ \nu=1,2,\ldots)
 \end{split}
\intertext{and therefore we may put}
 \begin{split}
  \alpha_0 &= e_0 = e_0^k - e_{0,1}-e_{1,1}-\dots-e_{k,1},\\
  \alpha_{j,\nu} &= e_{j,\nu} - e_{j,\nu+1}
  \qquad(j=0,\dots,k,\ \nu=1,2,\ldots).
 \end{split}
\end{align}
The element
\begin{equation}\label{eq:norm}
 \alpha_0(\ell_0,\dots,\ell_k)
 :=e_0^k - \sum_{j=0}^k\sum_{\nu=1}^{\ell_j+1}\frac{e_{j,\nu}}{\ell_j+1}
\end{equation}
is in the space spanned by $\alpha_0$ and $\alpha_{j,\nu}$ $(j=0,\dots,k,\,
\nu=1,\dots,\ell_j)$ and it is orthogonal to any $\alpha_{j,\nu}$ for 
$\nu=1,\dots,\ell_j$ and $j=0,\dots,k$. 

\begin{rem}\label{rem:norm}
We may assume $\ell_0\ge\ell_1\ge\cdots\ge\ell_k\ge 1$.
It is easy to have
\begin{align*}
  &\bigl(\alpha_0(\ell_0,\dots,\ell_k)|\alpha_0(\ell_0,\dots,\ell_k)\bigr)
 = 1 - k + \sum_{j=0}^k\frac1{\ell_j+1}\notag\\
 &\quad\begin{cases}
   >0 &(k=1)\\
   >0 &(k=2:\ell_1=\ell_2=1\text{ or } (\ell_0,\ell_1,\ell_2)=
    (2,2,1),\,(3,2,1)\text{ or }(4,2,1))\\
   =0 &(k=2:(\ell_0,\ell_1,\ell_2)=
    (2,2,2),\,(3,3,1)\text{ or }(5,2,1))\\
   <0 &(k=2:\ell_1\ge 2\text{ and }\
       \ell_0+2\ell_1+3\ell_2>12)\\
   = 0 &(k=3:\ell_0=\ell_1=\ell_2=\ell_3=1)\\
   < 0 &(k=3:\ell_0>1)\\
   < 0 &(k\ge 4)
  \end{cases}
\end{align*}
\end{rem}
The Weyl group $W_{\!\infty}$ of $\mathfrak g_\infty$ is the subgroup of 
$O(\mathfrak h)\subset GL(\mathfrak h)$ 
generated by the simple reflections
\begin{equation}\label{eq:Kzri}
 r_i(x) := x - 2\frac{(x|\alpha_i)}{(\alpha_i|\alpha_i)}\alpha_i
         = x - (x|\alpha_i)\alpha_i
 \qquad(x\in \mathfrak h,\ i\in I).
\end{equation}
The subgroup of $W_{\!\infty}$ generated by $r_i$ for $i\in I\setminus\{0\}$
is denoted by $W'_{\!\infty}$.
Putting $\sigma(\alpha_0)=\alpha_0$ and $\sigma(\alpha_{j,\nu})=
\alpha_{\sigma(j),\nu}$ for $\sigma\in\mathfrak S_\infty$, we define
a subgroup of $O(\mathfrak h)$:
\begin{equation}\label{eq:WS}
  \widetilde W_{\!\infty}:=\mathfrak S_\infty\ltimes W_{\!\infty}.
\end{equation}

For a tuple of partitions $\mathbf m
 =\bigl(m_{j,\nu}\bigr)_{j\ge 0,\ \nu\ge 1}
\in\mathcal P^{(n)}_{k+1}$ of $n$, we define
\begin{equation}\label{eq:Kazpart}
 \begin{split}
  n_{j,\nu}&:=m_{j,\nu+1}+m_{j,\nu+2}+\cdots,\\
  \alpha_{\mathbf m}&:=n\alpha_0
   + \sum_{j=0}^\infty\sum_{\nu=1}^\infty n_{j,\nu}\alpha_{j,\nu}
  = ne_0^k-\sum_{j=0}^\infty\sum_{\nu=1}^\infty m_{j,\nu}e_{j,\nu}
 \in Q_+.
 \end{split}
\end{equation}
\begin{prp}\label{prop:Kac}
{\rm i)} \ 
\hfill$\idx(\mathbf m,\mathbf m')=
 (\alpha_{\mathbf m}|\alpha_{\mathbf m'})$.\hfill\phantom{ABCDEFG}

\smallskip
{\rm ii)} \ 
Given $i\in I$, we have
$\alpha_{\mathbf m'} = r_i(\alpha_\mathbf m)$ with
\begin{equation*}
 \mathbf m'=
 \begin{cases}
   \p\mathbf m&(i=0),\\
   (m_{0,1}\dots,
   \overset{\underset{\smallsmile}1}m_{j,1}\dots
   \overset{\underset{\smallsmile}\nu}{m_{j,\nu+1}}
   \overset{\underset{\smallsmile}{\nu+1}}{m_{j,\nu}}\dots,\dots)
   &\bigl(i=(j,\nu)\bigr).
 \end{cases}
\end{equation*}
Moreover for $\ell=(\ell_0,\ell_1,\ldots)\in\mathbb Z_{>0}^\infty$
satisfying $\ell_\nu=1$ for $\nu\gg1$ we have
\begin{align}
 \alpha_\ell:=\alpha_{\mathbf 1_\ell} &=\alpha_0
     +\sum_{j=0}^\infty\sum_{\nu=1}^{\ell_j-1}\alpha_{j,\nu}
 =\biggl(\prod_{j\ge 0}
   r_{j,\ell_j-1}\cdots r_{j,2}r_{j,1}\biggr)(\alpha_0),\label{eq:alpl}
 \allowdisplaybreaks\\
 \alpha_{\p_\ell(\mathbf m)} &=
   \alpha_{\mathbf m} 
  - 2\frac{(\alpha_{\mathbf m}|\alpha_\ell)}
     {(\alpha_\ell|\alpha_\ell)}\alpha_\ell
  = \alpha_{\mathbf m}-(\alpha_m|\alpha_\ell)\alpha_\ell.\label{eq:M2K}
\end{align}
\end{prp}
\begin{proof} i) \ 
For a sufficiently large positive integer $k$ we have
\begin{align*}
 &\idx(\mathbf m,\mathbf m')
 =\sum_{j=0}^\infty\sum_{\nu=1}^\infty
 m_{j,\nu}m_{j,\nu}' - (k-1)\ord\mathbf m\cdot\ord\mathbf m'\\
 &\quad{}= \sum_{j=0}^k(n-n_{j,1})(n'-n'_{j,1})
    +\sum_{j=0}^k\sum_{\nu=1}^\infty
    (n_{j,\nu}-n_{j,\nu+1})(n'_{j,\nu}-n'_{j,\nu+1})
    - (k-1)nn'\\
 &\quad{}=2nn'+2\sum_{j=0}^k n_{j,\nu}n_{j,\nu}'
   -\sum_{j=0}^k(nn'_{j,1}+n'n_{j,1})
   -\sum_{j=0}^k\sum_{\nu=1}^\infty
    (n_{j,\nu}n_{j,\nu+1}'+n_{j,\nu}'n_{j,\nu+1})\\
 &\quad{}
  =(\alpha_{\mathbf m}|\alpha_{\mathbf m'}).
\end{align*}
The claim ii) easily follows from i).
\end{proof}
\begin{rem}
[\cite{Kc}]\label{rem:Kac}
The set $\Delta^{re}$ of \textsl{real roots} of the Kac-Moody Lie algebra 
equals $W_{\!\infty}\Pi$.
Denoting $K:=\{\beta\in Q_+\,;\,\supp\beta\text{ is connected and } 
(\beta,\alpha)\le 0
\quad(\forall\alpha\in\Pi)\}$, the set of \textsl{positive imaginary roots}
$\Delta^{im}_+$ equals $W_{\!\infty} K$.
The set $\Delta$ of roots equals 
$\Delta^{re}\cup\Delta^{im}$ 
by denoting
$\Delta_-^{im}=-\Delta_+^{im}$ and 
$\Delta^{im}=\Delta_+^{im}\cup\Delta_-^{im}$.
Put $\Delta_+=\Delta\cap Q_+$, $\Delta_-=-\Delta_+$.
Then $\Delta=\Delta_+\cup\Delta_-$ and
the root in $\Delta_+$ is called \textsl{positive}.
Here $\supp\beta=\{\alpha\in\Pi\,;\,n_\alpha\ne0\}$ if 
$\beta=\sum_{\alpha\in\Pi} n_\alpha\alpha$. 
A subset $L\subset\Pi$ is called \textsl{connected}
if the decomposition $L_1\cup L_2= L$ with 
$L_1\ne\emptyset$ and $L_2\ne\emptyset$ always implies the 
existence of $v_j\in L_j$ for $j=1$ and $2$ satisfying $(v_1|v_2)\ne0$.
\end{rem}
\begin{lmm}\label{lem:root}
{\rm i)\ } 
 Let $\alpha=n\alpha_0
+\displaystyle\sum_{j=0}^\infty\sum_{\nu=1}^\infty n_{j,\nu}\alpha_{j,\nu}\in
\Delta_+$ with $\supp\alpha\supsetneq\{\alpha_0\}$.  Then
\begin{gather}
  n\ge n_{j,1}\ge n_{j,2}\ge n_{j,3}\ge \cdots\qquad
  (j=0,1,\ldots),\\
  n\le \sum n_{j,1} - \max\{n_{j,1},n_{j,2},\ldots\}.\label{eq:root0}
\end{gather}

{\rm ii)\ } 
Let $\alpha=n\alpha_0
+\displaystyle\sum_{j=0}^\infty\sum_{\nu=1}^\infty n_{j,\nu}\alpha_{j,\nu}\in
Q_+$.
Suppose $\alpha$ is indivisible, that is,
$\frac 1k\alpha\notin Q$ for $k=2,3,\ldots$.
Then $\alpha$ corresponds to a basic tuple if and only if
\begin{equation}
 \left\{
 \begin{aligned}
  2n_{j,\nu}&\le n_{j,\nu-1}+n_{j,\nu+1}
  \quad(n_{j,0}=n,\ j=0,1,\ldots,\ \nu=1,2,\ldots),\\
  2n&\le n_{0,1}+n_{1,1}+n_{2,1}+\cdots.
 \end{aligned}\right.
\end{equation}
\end{lmm}
\begin{proof} The lemma is clear from the following for 
$\alpha=n\alpha_0+\sum n_{j,\nu}\alpha_{j,\nu}\in\Delta_+$:
\begin{align}
 r_{i,\mu}(\alpha)&=n\alpha_0 + 
 \sum \bigl(n_{j,\nu}-\delta_{i,j}\delta_{\mu,\nu}
  (2n_{j,\mu}- n_{j,\mu-1}- n_{j,\mu+1})\bigr)\alpha_{j,\nu}
 \in\Delta,\label{eq:An}\\
 r_0(\alpha)&=\bigl(\sum n_{j,1}-n\bigr)\alpha_0
 +\sum n_{j,\nu}\alpha_{j,\nu}\in\Delta.\label{eq:r0}
\end{align}
For example, 
putting $n_{j,0}=n>0$ and 
$r_{i,N}\cdots r_{i,\mu+1}r_{i,\mu}\alpha
=n\alpha_0+\sum n'_{j,\nu}\alpha_{j,\nu}\in\Delta_+$ for a sufficiently large $N$, we have $n'_{j,N}= n_{j,N}+n_{j,\mu-1}-n_{j,\mu}=n_{j,\mu-1}-n_{j,\mu}\ge 0$ for $\mu=1,2,\ldots$ and moreover \eqref{eq:root0} by $r_0\alpha\in\Delta_+$.
\end{proof}

\begin{rem}\label{rem:KM}
{\rm i)\ } 
It follows from \eqref{eq:M2K} that Katz's middle convolution corresponds
to the reflection with respect to the root $\alpha_\ell$
under the identification $\mathcal P\subset Q_+$ with \eqref{eq:Kazpart}.

Moreover there is a natural correspondence between the set of irreducibly 
realizable tuples of partitions and the set of positive roots $\alpha$ of 
$\mathfrak g_\infty$ with $\supp\alpha\ni\alpha_0$ such that 
$\alpha$ is indivisible or $(\alpha|\alpha)<0$.
Then the rigid tuple of partitions corresponds to the positive real root 
whose support contains $\alpha_0$.

\medskip
\begin{tabular}{|c|c|}\hline
$\mathcal P$
 & Kac-Moody root system\\ \hline\hline
$\mathbf m$ 
 & $\alpha_{\mathbf m}$ (cf.~\eqref{eq:Kazpart})\\ \hline
 $\mathbf m$ : rigid
 & $\alpha\in\Delta_+^{re}\,:\,\supp\alpha\ni\alpha_0$\\ \hline
 $\mathbf m$ : basic 
 & $\alpha\in Q_+$\,:\, $(\alpha|\beta)\le 0
  \ \ (\forall\beta\in\Pi)$\\ 
 (cf.~\eqref{eq:basic})
 & indivisible and $\supp\alpha$ is connected\\ \hline
 \raisebox{-6.5pt}{$\mathbf m$ : irreducibly realizable}
 &$\alpha\in\Delta_+$ : $\supp\alpha\ni\alpha_0$\\[-4pt] 
 & indivisible or $(\alpha|\alpha)<0$\\ \hline
$\ord\mathbf m$ & $n\,:\,\alpha
 =n\alpha_0+\sum_{j,\nu}n_{j,\nu}\alpha_{j,\nu}$\\ \hline
$\idx(\mathbf m,\mathbf m')$ 
& $(\alpha_{\mathbf m}|\alpha_{\mathbf m'})$\\ \hline
$\Pidx\mathbf m+\Pidx\mathbf m'=\Pidx(\mathbf m+\mathbf m')$
 &$(\alpha_{\mathbf m}|\alpha_{\mathbf m'})=-1$\\ \hline
$(\nu,\nu+1)\in G_j\subset S_\infty'$ (cf.~\eqref{eq:S_infty})
 & 
 $s_{j,\nu}\in W_{\!\infty}'$ (cf.~\eqref{eq:Kzri})\\ \hline
$\p$ in \eqref{eq:p} & $r_0$ in \eqref{eq:r0}\\ \hline
$H\simeq \mathfrak S_\infty$ (cf.~\eqref{eq:S_infty})
 & $\mathfrak S_\infty$ in \eqref{eq:WS}\\ \hline 
$\langle\p,\,S_\infty\rangle$ (cf.~Definition~\ref{dfn:Sinfty})
 & $\widetilde W_{\!\infty}$ in \eqref{eq:WS}\rule{0pt}{11.5pt}\\ \hline
\end{tabular}
\medskip

Here we define $\Pidx\mathbf m:=1-\tfrac12\idx\mathbf m$
as in Definition~\ref{defn:pidx} and $\langle\p,\,S_\infty\rangle$
denotes the group generated by $\p$ and $S_\infty$.
\smallskip

{\rm ii)\ }
For an irreducibly realizable tuple $\mathbf m\in\mathcal P$,
$\p(\mathbf m)$ is well-defined if and only if
$\ord\mathbf m>1$ or $\sum_{j=0}^\infty m_{j,2} > 1$,
which corresponds to the condition \eqref{eq:star}.

\smallskip
{\rm iii)\ }
Suppose a tuple $\mathbf m\in\mathcal P_{k+1}^{(n)}$ is basic.
The subgroup of $W_{\!\infty}$ generated by reflections with
respect to $\alpha_\ell$ (cf.~\eqref{eq:alpl}) 
satisfying $(\alpha_{\mathbf m}|\alpha_\ell)=0$ and $\supp\alpha_\ell
\subset\supp\mathbf\alpha_{\mathbf m}$ is infinite if and only if 
$\idx\mathbf m=0$.

Note that the condition $(\alpha_{\mathbf m}|\alpha_\ell)=0$
means that the corresponding middle convolution of
$\mathbf A\in M(n,\mathbb C)^{k+1}_0$ with 
$\spt\mathbf A=\mathbf m$ keeps the partition type invariant.
%
\end{rem}
\begin{prp}
For irreducibly realizable $\mathbf m\in\mathcal P$ 
and $\mathbf m'\in\mathcal R$ satisfying
\begin{equation}\label{eq:nrigid}
\ord \mathbf m>\idx(\mathbf m,\mathbf m')\cdot\ord\mathbf m',
\end{equation}
we have
\begin{gather}
  \mathbf m'':=\mathbf m-\idx(\mathbf m,\mathbf m')\mathbf m'
  \text{ is irreducibly realizable},\allowdisplaybreaks\\
  \idx\mathbf m''=\idx\mathbf m.
\end{gather}
Here \eqref{eq:nrigid} is always valid if $\mathbf m$ is not rigid.
\end{prp}
\begin{proof}
The claim follows from the fact that $\alpha_{m''}$ is the reflection 
of the root $\alpha_{\mathbf m}$ with respect to the real root 
$\alpha_{\mathbf m'}$.
\end{proof}
\section{A classification of tuples of partitions}
In this section we assume 
that a $(k+1)$-tuple $\mathbf m=\bigl(m_{j,\nu}\bigr)_{\substack{0\le j\le k\\1\le \nu\le n_j}}$ 
of partitions of a positive integer satisfies
\begin{equation}
 m_{j,1}\ge m_{j,2}\ge \cdots \ge m_{j,n_j}\ge 1\quad\text{and}\quad
 n_j\ge 2\quad(j=0,1,\dots,k).
\end{equation}
Note that
\[
 m_{j,1}+m_{j,2}+\cdots+m_{j,n_j}=\ord\mathbf m\ge 2
 \quad(j=0,1,\dots,k).
\]
\begin{prp}\label{prop:basic}
Let $\mathcal K$ denote the totality of basic elements of $\mathcal P$
defined in \eqref{eq:basic} and for an even integer $p$ put
\begin{equation*}
  \mathcal K(p)
  :=\{\mathbf m\in\mathcal K\,;\,\idx\mathbf m=p\}.
\end{equation*}
Then $\#\mathcal K(p)<\infty$.
In particular $\mathcal K(p)=\emptyset$ if $p>0$ and
\begin{gather}
 \mathcal K\cap\langle\p,S'_{\infty}\rangle\mathbf m=\{\mathbf m\}
 \quad(\mathbf m\in\mathcal K),
 \label{eq:indep}\\
\bar{\mathcal K}(0)=
\left\{
 11,11,11,11\ \  
 111,111,111\ \ 
 22,1111,1111\ \ 
 33,222,111111
 \right\}.\label{eq:listidx0}
\end{gather}
Here we use the notation in Remark~\ref{rem:KM} i), 
$\bar{\mathcal K}(p)$ denotes the quotient of $\mathcal K(p)$
under the action of the group $S_\infty$ and the element of 
$\bar{\mathcal K}(p)$ is denoted by its representative.
\end{prp}
\begin{proof}
It follows from Remark~\ref{rem:KM} i) that $\mathcal K$ corresponds to 
the set of indivisible roots in $K$ in Remark~\ref{rem:Kac}
and we have \eqref{eq:indep} because $K\cap W_{\!\infty}\alpha=\{\alpha\}$
for $\alpha\in K$.

Let $\mathbf m\in\mathcal K\cap {\mathcal P}_{k+1}$.
We may assume that $\mathbf m$ is monotone and indivisible.
Since
\begin{equation}\label{eq:idxn}
 \idx\mathbf m +  \sum_{j=0}^k\sum_{\nu=2}^{n_j}
 (m_{j,1}-m_{j,\nu})\cdot m_{j,\nu}
=\biggl(\sum_{j=0}^km_{j,1}-(k-1)\ord\mathbf m\biggr)
 \cdot\ord\mathbf m,
\end{equation}
the assumption $\mathbf m\in\mathcal K$ is equivalent to
\begin{equation}\label{eq:idx01}
 \sum_{j=0}^k\sum_{\nu=2}^{n_j}
 (m_{j,1}-m_{j,\nu})\cdot m_{j,\nu}\le -\idx\mathbf m.
\end{equation}
Hence $\idx\mathbf m\le 0$.

First suppose $\idx\mathbf m=0$.
Then 
$m_{j,1}=m_{j,2}=\cdots=m_{j,n_j}$ and the identity
\begin{equation}\label{eq:idxid}
\sum_{j=0}^k\frac{m_{j,1}}{\ord\mathbf m}
 = k-1 + \frac{\idx\mathbf m}{(\ord\mathbf m)^2}
  +  \sum_{j=0}^k\sum_{\nu=1}^{n_j}
  \frac{(m_{j,1}-m_{j,\nu})m_{j,\nu}}{(\ord\mathbf m)^2}
\end{equation}
implies
$
  \sum_{j=0}^k\frac 1{n_j}=k-1
$.
Since $\sum_{j=0}^k\frac 1{n_j}\le \frac{k+1}2$, we have
$k\le 3$.
When $k=3$, we have $n_0=n_1=n_2=n_3=2$.
When $k=2$, 
$ 
  \frac 1{n_0}+\frac 1{n_1}+\frac1{n_2}=1
$ 
and we easily conclude that $\{n_0,n_1,n_2\}$ equals
$\{3,3,3\}$ or $\{2,4,4\}$ or $\{2,3,6\}$, which means
\eqref{eq:listidx0}.

Since 
$\idx\mathbf m=2(\ord\mathbf m)^2 - \sum_{j=0}^k N_j$ with
$N_j=(\ord\mathbf m)^2 - \sum_{\nu=0}^{n_j}m_{j,\nu}^2>0$, 
there exist a finite number of $\mathbf m\in\mathcal P$ such that
the numbers $\ord\mathbf m$ and $\idx\mathbf A$ are fixed 
because $k$ is bounded.
Therefore to prove the remaining part of the lemma we may assume
\begin{equation}\label{eq:AsL}
 \idx\mathbf m\le -2\quad\text{and}\quad
  \ord\mathbf m \ge -7\idx\mathbf m+7.
\end{equation}
Then
\begin{equation}\label{eq:AsL2}
 \ord\mathbf m \ge 21\quad\text{and}\quad
 (\ord\mathbf m)^2 > -147\idx\mathbf m.
\end{equation}
If $m_{j,1}>m_{j,n_j}>0$,
\eqref{eq:idx01} implies $m_{j,1}-1\le -\idx\mathbf m
\le\frac17\ord\mathbf m-1$ and therefore 
\begin{gather}
m_{j,1}\le\frac17\ord\mathbf m,\\
\sum_{\nu=1}^{n_j}m_{j,\nu}^2\le m_{j,1}\cdot\ord\mathbf m
\le\frac17(\ord\mathbf m)^2.\label{eq:1over7}
\end{gather}

Hence $2m_{j,1}\le\ord\mathbf m$ for $j=0,\dots,k$,
\begin{align*}
 \idx\mathbf m+(k-1)\cdot(\ord\mathbf m)^2
 =\sum_{j=0}^k\sum_{\nu=1}^{n_j}m_{j,\nu}^2
 \le \sum_{j=0}^k\frac12(\ord\mathbf m)^2
 =\frac{k+1}2(\ord\mathbf m)^2
\end{align*}
and
$
 \frac{k-3}{2}(\ord\mathbf m)^2
 \le-\idx\mathbf m<\frac17\ord\mathbf m$,
which proves $k\le 3$.

Suppose $k=3$.
Since $\mathbf m\ne 11,11,11,11$, we have
$m_{j,1}\le \frac13\ord\mathbf m$ with a suitable $j$,
\begin{align*}
 \idx\mathbf m &= \sum_{j=0}^3\sum_{\nu=1}^{n_j}m_{j,\nu}^2
 -2\cdot(\ord\mathbf m)^2
 \le\sum_{j=0}^3m_{j,1}\ord\mathbf m - 2(\ord\mathbf m)^2\\
 &\le(\tfrac12+\tfrac12+\tfrac12+\tfrac13-2)(\ord\mathbf m)^2
 =-\tfrac16(\ord\mathbf m)^2
\end{align*}
and $\ord\mathbf m\le -\frac{6\idx\mathbf m}{\ord\mathbf m}
\le-\frac27\idx\mathbf m$, which contradicts to \eqref{eq:AsL}.

Suppose $k=2$ and put $J=\{j\,;\,m_{j,1}\ne m_{j,n_j}\quad(j=0,1,2)\}$.
Then
\[
 1+\frac{\idx\mathbf m}{(\ord\mathbf m)^2}
 =
 \frac{\sum_{\nu=1}^{n_0}m_{0,\nu}^2}
 {(\ord\mathbf m)^2}
 +\frac{\sum_{\nu=1}^{n_1}m_{1,\nu}^2}
 {(\ord\mathbf m)^2}
 +\frac{\sum_{\nu=1}^{n_2}m_{2,\nu}^2}
 {(\ord\mathbf m)^2}
\]
and therefore
\[
  1-\frac1{147}-\frac{\#J}7<
  \sum_{j\in\{0,1,2\}\setminus J}\frac1{n_j}< 1
\]
because of \eqref{eq:AsL}, \eqref{eq:AsL2} and \eqref{eq:1over7} for $j\in J$.
Lemma~\ref{lem:fracsum} assures that this never holds because
$1-\frac1{147}-\frac37>0$,
$1-\frac1{147}-\frac27>\frac12$,
$1-\frac1{147}-\frac17>\frac56$ and
$1-\frac1{147}>\frac{41}{42}$
according to $\#J=3,2,1$ and $0$, respectively.
\end{proof}
\begin{lmm}\label{lem:fracsum}
Put
$I_{k+1}=\left\{\sum_{j=0}^k\frac1{n_j}\,;\,
 n_j\in\{2,3,4,\dots\}\right\}\cap[0,1)$. Then
\[
 I_1\subset (0,\tfrac12],\ I_2\subset(0,\tfrac56]\text{ \ and \ }
 I_3\subset(0,\tfrac{41}{42}].
\]
\end{lmm}
\begin{proof}
Let $r\in I_{k+1}$.
It is clear that $r\le\frac12$ for $r\in I_1$.

Let $r=\frac1{n_0}+\frac1{n_1}\in I_2$.
If $n_0=2$, then $n_1\ge 3$ and $r\le\frac56$.
If $n_0\ge 3$, then $r\le\frac23$.

Let $r=\frac1{n_0}+\frac1{n_1}+\frac1{n_2}\in I_3$.
We may assume $n_0\le n_1\le n_2$.

If $n_0\le 4$, then $r\le\frac34$.

Suppose $n_0=3$.
If $n_1\ge 4$, $r\le\frac56$.
If $n_1=3$, then $n_2\ge 4$ and $r\le\frac{11}{12}$.

Suppose $n_0=2$.  Then $n_1\ge 3$.
If $n_1=3$, then $n_2>6$ and $r\le\frac{41}{42}$.
If $n_1\ge 4$, then $n_2>4$ and $r\le\frac{19}{20}$. 
\end{proof}

\begin{rem}
{\rm i)\ }
$\bar{\mathcal K}(0)$ is given by Kostov \cite{Ko2} and its elements correspond
to the indivisible positive null-roots $\alpha$ of the affine root systems 
$\tilde D_4$, $\tilde E_6$, $\tilde E_7$ and $\tilde E_8$ 
(cf.~Remark~\ref{rem:norm}, Proposition~\ref{prop:Kac} and 
Table $\bar{\mathcal K}(0)$).

{\rm ii)\ }
In the proof we obtained
$\ord\mathbf m + 7\idx\mathbf m\le 6$ for $\mathbf m\in\mathcal K$ but
we can prove
\begin{align}
\ord\mathbf m + 3\idx\mathbf m&\le 6\text{ \ for \ }\mathbf m\in\mathcal K,\\
\ord\mathbf m + \idx\mathbf m&\le 2\text{ \ for \ }
 \mathbf m\in\mathcal K\setminus\mathcal P_3.
\end{align}
\end{rem}
\begin{exa}\label{ex:special}
For a positive integer $m$ we have special 4 elements
\begin{equation}
\begin{aligned}\label{eq:Qsp}
 &D_4^{(m)}: mm-11,m^2,m^2,m^2&& 
 \quad E_6^{(m)}: m^2m-11,m^3,m^3\\ 
 &E_7^{(m)}: m^3m-11,m^4,(2m)^2&& 
 \quad E_8^{(m)}: m^5m-11,(2m)^3,(3m)^2
\end{aligned}
\end{equation}
in $\bar{\mathcal K}(2-2m)$ with orders $2m$, $3m$, $4m$ and $6m$, 
respectively.
\end{exa}
\begin{prp} We have
\begin{align*}
  \bar{\mathcal K}(-2)&=\bigl\{
    11,11,11,11,11
\ \ 21,21,111,111
\ \ 31,22,22,1111
\ \ 22,22,22,211
\\&\qquad
   211,1111,1111
\ \ 221,221,11111
\ \ 32,11111,11111
\ \ 222,222,2211
\\&\qquad
    33,2211,111111
\ \ 44,2222,22211
\ \ 44,332,11111111
\ \ 55,3331,22222
\\&\qquad
    66,444,2222211
\bigr\}.
\end{align*}
\end{prp}
\begin{proof}
Let $\mathbf m\in\mathcal K(-2)\cap{\mathcal P}_{k+1}$ be 
monotone. Then \eqref{eq:idx01} and \eqref{eq:idxn} with 
$\idx\mathbf m=-2$ implies
$\sum (m_{j,1}-m_{j,\nu})m_{j,\nu}=0$ or $2$ 
and we have
the following 5 possibilities.

(A) \ $m_{0,1}\cdots m_{0,n_0} = 2\cdots211$
      and $m_{j,1}=m_{j,n_j}$ for $1\le j\le k$.

(B) \ $m_{0,1}\cdots m_{0,n_0} = 3\cdots31$
    and $m_{j,1}=m_{j,n_j}$ for $1\le j\le k$.

(C) \ $m_{0,1}\cdots m_{0,n_0} = 3\cdots32$
    and $m_{j,1}=m_{j,n_j}$ for $1\le j\le k$.

(D) \ $m_{i,1}\cdots m_{i,n_0} = 2\cdots21$
     and $m_{j,1}=m_{j,n_j}$ 
     for $0\le i\le 1<j\le k$.

(E) \ $m_{j,1}=m_{j,n_j}$ for $0\le j\le k$ and $\ord\mathbf m=2$.

\smallskip
Case (A). \ 
If $2\cdots211$ is replaced by $2\cdots22$, $\mathbf m$
is transformed into $\mathbf m'$ with $\idx\mathbf m'=0$.
If $\mathbf m'$ is indivisible, $\mathbf m'\in\mathcal K(0)$ and
$\mathbf m$ is $211,1^4,1^4$ or $33,2211,1^6$.
If $\mathbf m'$ is not indivisible, $\frac12\mathbf m'\in\mathcal K(0)$ and
$\mathbf m$ is one of the tuples given in \eqref{eq:Qsp} with $m=2$.

Put $m=n_0-1$ and examine the identity \eqref{eq:idxid}.

\smallskip
Case (B). \ 
$\frac3{3m+1}+\frac1{n_1}+\cdots+\frac1{n_k}
=k-1$.
Since $n_j\ge 2$, we have $\frac12k -1\le\frac{3}{3m+1}<1$
and $k\le 3$.

If $k=3$, we have $m=1$, $\ord\mathbf m=4$,
$\frac1{n_1}+\frac1{n_2}+\frac1{n_3}=\frac54$,
$\{n_1,n_2,n_3\}=\{2,2,4\}$ and 
$\mathbf m=31,22,22,1111$.

Assume $k=2$.  Then
$\frac1{n_1}+\frac1{n_2}=1-\frac{3}{3m+1}$ and Lemma~\ref{lem:fracsum}
implies $m\le 5$.  We have
$\frac1{n_1}+\frac1{n_2}=\frac{13}{16}$, $\frac{10}{13}$, $\frac{7}{10}$,
$\frac{4}{7}$ and $\frac{1}{4}$
according to $m=5$, $4$, $3$, $2$ and $1$, respectively.
Hence we have $m=3$, $\{n_1,n_2\}=\{2,5\}$ and $\mathbf m=3331,55,22222$.

\smallskip
Case (C). \ 
$\frac{3}{3m+2}+\frac1{n_1}+\cdots+\frac1{n_k}=k-1$.
Since $n_j\ge 2$, $\frac12k-1\le \frac{3}{3m+2}<1$
and $k\le 3$.
If $k=3$, then $m=1$, $\ord\mathbf m=5$ and 
$\frac1{n_1}+\frac1{n_2}+\frac1{n_3}=\frac75$, which never occurs.

Thus we have $k=2$, 
$\frac1{n_1}+\frac1{n_2}=1-\frac{3}{3m+2}$ and Lemma~\ref{lem:fracsum}
implies $m\le 5$.
We have $\frac1{n_1}+\frac1{n_2}=\frac{14}{17}$, 
$\frac{11}{14}$, $\frac{8}{11}$, $\frac{5}{8}$ 
and $\frac{2}{5}$ according to $m=5$, $4$, $3$, $2$ and $1$, respectively.
Hence we have $m=1$ and $n_1=n_2=5$ and
$\mathbf m=32,11111,11111$
or $m=2$ and $n_1=2$ and $n_2=8$ and $\mathbf m=332,44,11111111$.

\smallskip
Case (D). \ 
$\frac2{2m+1}+\frac2{2m+1}+\frac1{n_2}+\cdots+\frac1{n_k}=k-1$.
Since $n_j\ge3$ for $j\ge2$, we have
$k-1\le \frac32\frac{4}{2m+1}=\frac{6}{2m+1}$ and 
$m\le2$.
If $m=1$, then $k\le 3$ and $\frac1{n_2}+\frac1{n_3}=2-\frac43=\frac23$ and 
we have $\mathbf m=21,21,111,111$.
If $m=2$, then $k=2$, $\frac1{n_2}=1-\frac{4}{5}$ and
$\mathbf m=221,221,11111$.

\smallskip
Case (E). \ 
Since $m_{j,1}=1$ and \eqref{eq:idxn} means $-2=\sum_{j=0}^k 2m_{j,1}-4(k-1)$, 
we have $k=4$ and $\mathbf m=11,11,11,11,11$.
\end{proof}
By the aid of a computer we have the following tables.

\begingroup
\samepage
\centerline{Table of $\#\bar{\mathcal K}(p)$.}
\smallskip
\noindent
\begin{tabular}{|c|r|r|r|r|r|r|r|r|r|r|r|}\hline
index         & \ \,0\ & $-2$ & $-4$ & $-6$ & $-8$ & $-10$ & $-12$ & $-14$ & $-16$ & $-18$ & $-20$
\\ \hline   
$\#\bar{\mathcal K}(p)$
& \ \,4\ & 13 & 36 & 67 & 90 & 162 & 243 & 305 & 420 & 565 & 720
\\ \hline
$\#$ triplets & \ \,3\ & \ \,9&24& 44 & 56 & \ \,97& 144 & 163 & 223 & 291 & 342
\\ \hline
$\#$ 4-tuples & \ \,1\ & \ \,3& \ \,9& 17 & \ \,24 & \ \,45& \ \,68 & \ \,95 & 128 & 169 & 239
\\ \hline
\end{tabular}
\endgroup
\medskip

\if\RIMS
\noindent
{Table of 
$(\ord\mathbf m:\mathbf m)$ of $\bar{\mathcal K}(-4)$ ($*$ 
corresponds to \eqref{eq:Qsp} and $+$ means 
$\p_{max}(\mathbf m)\ne\mathbf m$))}
\vspace{-6pt}
\else
\centerline{Table of 
$(\ord\mathbf m:\mathbf m)$ of $\bar{\mathcal K}(-4)$ ($*$  
: \eqref{eq:Qsp}\quad $+$ : $\p_{max}(\mathbf m)\ne\mathbf m$))}
\fi
{\small\begin{verbatim}
 +2:11,11,11,11,11,11      3:111,21,21,21,21         4:22,22,22,31,31
 +3:111,111,111,21        +4:1111,22,22,22           4:1111,1111,31,31
  4:211,211,22,22          4:1111,211,22,31         *6:321,33,33,33
  6:222,222,33,51         +4:1111,1111,1111          5:11111,11111,311
  5:11111,2111,221         6:111111,222,321          6:111111,21111,33
  6:21111,222,222          6:111111,111111,42        6:222,33,33,42
  6:111111,33,33,51        6:2211,2211,222           7:1111111,2221,43
  7:1111111,331,331        7:2221,2221,331           8:11111111,3311,44
  8:221111,2222,44         8:22211,22211,44         *9:3321,333,333
  9:111111111,333,54       9:22221,333,441          10:1111111111,442,55 
 10:22222,3322,55         10:222211,3331,55         12:22221111,444,66
*12:33321,3333,66         14:2222222,554,77        *18:3333321,666,99
\end{verbatim}}

We express the root $\alpha_{\mathbf m}$ for 
$\mathbf m\in\bar{\mathcal K}(0)$ and $\bar{\mathcal K}(-2)$ 
using Dynkin diagram.
The circles in the diagram represent the simple roots in 
$\supp\alpha_{\mathbf m}$ and two circles are connected by a 
line if the inner product of the corresponding simple roots is not zero.
The number attached to a circle is the corresponding coefficient 
$n$ or $n_{j,\nu}$ in the expression \eqref{eq:Kazpart}.

For example, if $\mathbf m=11,11,11,11$, then 
$\alpha_{\mathbf m}=2\alpha_0+\alpha_{0,1}+\alpha_{1,2}
+\alpha_{2,2}+\alpha_{3,2}$, 
which corresponds to the first diagram in the following.
\smallskip

\if\RIMS
\else
\bigskip
\bigskip
\fi
\centerline{{\bf Table} \ $\bar{\mathcal K}(0)$}
\vspace{-.6cm}
\nopagebreak
{\small
\begin{gather*}
\begin{xy}
\ar@{-}               *++!D{1}  *\cir<4pt>{};
             (10,0)   *+!L+!D{2}*\cir<4pt>{}="A",
\ar@{-} "A"; (20,0)   *++!D{1}  *\cir<4pt>{},
\ar@{-} "A"; (10,-10) *++!L{1}  *\cir<4pt>{},
\ar@{-} "A"; (10,10)  *++!L{1}  *\cir<4pt>{},
\end{xy}
\quad
\begin{xy}
\ar@{-}               *++!D{2}  *\cir<4pt>{};
             (10,0)   *++!D{4}  *\cir<4pt>{}="A",
\ar@{-} "A"; (20,0)   *+!L+!D{6}*\cir<4pt>{}="B",
\ar@{-} "B"; (30,0)   *++!D{5}  *\cir<4pt>{}="C",
\ar@{-} "C"; (40,0)   *++!D{4}  *\cir<4pt>{}="D",
\ar@{-} "D"; (50,0)   *++!D{3}  *\cir<4pt>{}="E",
\ar@{-} "E"; (60,0)   *++!D{2}  *\cir<4pt>{}="F",
\ar@{-} "F"; (70,0)   *++!D{1}  *\cir<4pt>{},
\ar@{-} "B"; (20,10)  *++!L{3}  *\cir<4pt>{}
\end{xy}
\allowdisplaybreaks\\[-1.2cm]
\begin{xy}
\ar@{-}               *++!D{1}  *\cir<4pt>{};
             (10,0)   *++!D{2}  *\cir<4pt>{}="A",
\ar@{-} "A"; (20,0)   *++!D{3}  *\cir<4pt>{}="B",
\ar@{-} "B"; (30,0)   *+!L+!D{4}*\cir<4pt>{}="C",
\ar@{-} "C"; (40,0)   *++!D{3}  *\cir<4pt>{}="D",
\ar@{-} "D"; (50,0)  *++!D{2}   *\cir<4pt>{}="E",
\ar@{-} "E"; (60,0)  *++!D{1}   *\cir<4pt>{},
\ar@{-} "C"; (30,10)  *++!L{2}  *\cir<4pt>{}
\end{xy}
\quad
\begin{xy}
\ar@{-}               *++!D{1}  *\cir<4pt>{};
             (10,0)   *++!D{2}  *\cir<4pt>{}="O",
\ar@{-} "O"; (20,0)   *+!L+!D{3}*\cir<4pt>{}="A",
\ar@{-} "A"; (30,0)   *++!D{2}  *\cir<4pt>{}="B",
\ar@{-} "B"; (40,0)   *++!D{1}  *\cir<4pt>{}
\ar@{-} "A"; (20,10)  *++!L{2}  *\cir<4pt>{}="C",
\ar@{-} "C"; (20,20)  *++!L{1}  *\cir<4pt>{},
\end{xy}
\end{gather*}}
\centerline{{\bf Table } $\bar{\mathcal K}(-2)$}
\centerline{Dotted circles represent simple roots which are not orthogonal
to the root.}

\vspace{-.6cm}
\nopagebreak
{\small
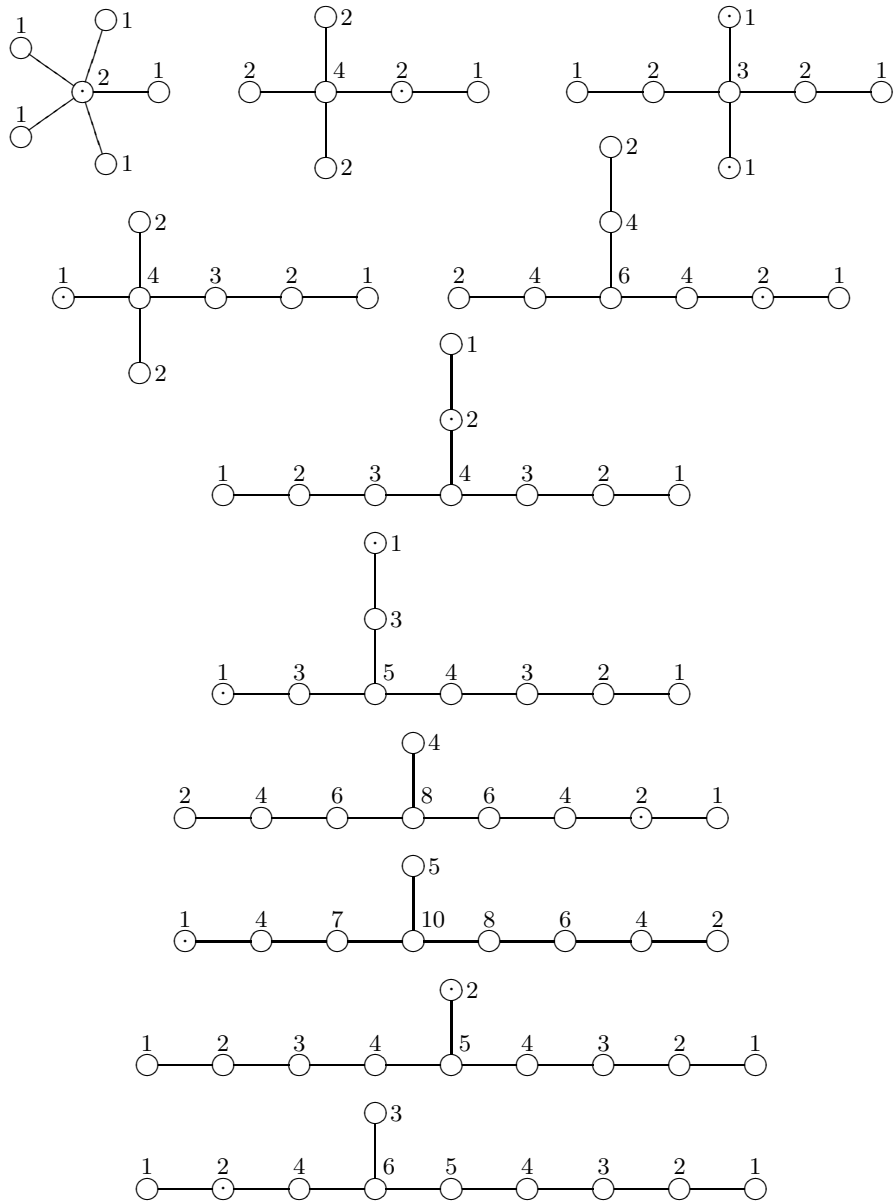
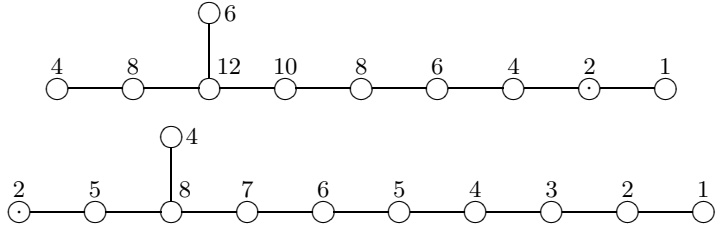
\begin{gather*}
\begin{xy}
\ar@{-}      (10,0)       *++!D{1}  *\cir<4pt>{};
             (0,0)        *+!D+!L{2}*{\cdot}*\cir<4pt>{}="A",
\ar@{-} "A"; (3.09,9.51)  *++!L{1}  *\cir<4pt>{};
\ar@{-} "A"; (-8.09,5.88) *++!D{1}  *\cir<4pt>{};
\ar@{-} "A"; (3.09,-9.51) *++!L{1}  *\cir<4pt>{};
\ar@{-} "A"; (-8.09,-5.88)*++!D{1}  *\cir<4pt>{};
\end{xy}\qquad
\begin{xy}
\ar@{-}               *++!D{2}  *\cir<4pt>{};
             (10,0)   *+!L+!D{4}*\cir<4pt>{}="A",
\ar@{-} "A"; (20,0)   *++!D{2}  *{\cdot} *\cir<4pt>{}="B",
\ar@{-} "B"; (30,0)   *++!D{1}  *\cir<4pt>{},
\ar@{-} "A"; (10,-10) *++!L{2}  *\cir<4pt>{},
\ar@{-} "A"; (10,10)  *++!L{2}  *\cir<4pt>{}
\end{xy}\qquad\ 
\begin{xy}
\ar@{-}               *++!D{1}  *\cir<4pt>{};
             (10,0)   *++!D{2}  *\cir<4pt>{}="A",
\ar@{-} "A"; (20,0)   *+!L+!D{3}*\cir<4pt>{}="B",
\ar@{-} "B"; (30,0)   *++!D{2}  *\cir<4pt>{}="C",
\ar@{-} "C"; (40,0)   *++!D{1}  *\cir<4pt>{},
\ar@{-} "B"; (20,-10) *++!L{1}  *{\cdot}*\cir<4pt>{}="C",
\ar@{-} "B"; (20,10)  *++!L{1}  *{\cdot}*\cir<4pt>{}
\end{xy}
\allowdisplaybreaks\\[-1cm]
\begin{xy}
\ar@{-}               *++!D{1}  *{\cdot}*\cir<4pt>{};
             (10,0)   *+!L+!D{4}*\cir<4pt>{}="A",
\ar@{-} "A"; (20,0)   *++!D{3}  *\cir<4pt>{}="B",
\ar@{-} "B"; (30,0)   *++!D{2}  *\cir<4pt>{}="C",
\ar@{-} "C"; (40,0)   *++!D{1}  *\cir<4pt>{},
\ar@{-} "A"; (10,-10) *++!L{2}  *\cir<4pt>{},
\ar@{-} "A"; (10,10)  *++!L{2}  *\cir<4pt>{}
\end{xy}\qquad
\begin{xy}
\ar@{-}               *++!D{2}  *\cir<4pt>{};
             (10,0)   *++!D{4}  *\cir<4pt>{}="A",
\ar@{-} "A"; (20,0)   *+!L+!D{6}*\cir<4pt>{}="B",
\ar@{-} "B"; (30,0)   *++!D{4}  *\cir<4pt>{}="C",
\ar@{-} "C"; (40,0)   *++!D{2}  *{\cdot}*\cir<4pt>{}="D",
\ar@{-} "D"; (50,0)   *++!D{1}  *\cir<4pt>{},
\ar@{-} "B"; (20,10)  *++!L{4}  *\cir<4pt>{}="F",
\ar@{-} "F"; (20,20)  *++!L{2}  *\cir<4pt>{},
\end{xy}\allowdisplaybreaks\\[-1.1cm]
\begin{xy}
\ar@{-}               *++!D{1}  *\cir<4pt>{};
             (10,0)   *++!D{2}  *\cir<4pt>{}="A",
\ar@{-} "A"; (20,0)   *++!D{3}  *\cir<4pt>{}="B",
\ar@{-} "B"; (30,0)   *+!L+!D{4}*\cir<4pt>{}="C",
\ar@{-} "C"; (40,0)   *++!D{3}  *\cir<4pt>{}="D",
\ar@{-} "D"; (50,0)   *++!D{2}  *\cir<4pt>{}="E",
\ar@{-} "E"; (60,0)   *++!D{1}  *\cir<4pt>{},
\ar@{-} "C"; (30,10)  *++!L{2}  *{\cdot}*\cir<4pt>{}="F",
\ar@{-} "F"; (30,20)  *++!L{1}  *\cir<4pt>{},
\end{xy}
\allowdisplaybreaks\\
\begin{xy}
\ar@{-}               *++!D{1}  *{\cdot}*\cir<4pt>{};
             (10,0)   *++!D{3}  *\cir<4pt>{}="A",
\ar@{-} "A"; (20,0)   *+!L+!D{5}*\cir<4pt>{}="B",
\ar@{-} "B"; (30,0)   *++!D{4}  *\cir<4pt>{}="C",
\ar@{-} "C"; (40,0)   *++!D{3}  *\cir<4pt>{}="D",
\ar@{-} "D"; (50,0)   *++!D{2}  *\cir<4pt>{}="E",
\ar@{-} "E"; (60,0)   *++!D{1}  *\cir<4pt>{},
\ar@{-} "B"; (20,10)  *++!L{3}  *\cir<4pt>{}="F",
\ar@{-} "F"; (20,20)  *++!L{1}  *{\cdot}*\cir<4pt>{},
\end{xy}\allowdisplaybreaks\\
\begin{xy}
\ar@{-}               *++!D{2}  *\cir<4pt>{};
             (10,0)   *++!D{4}  *\cir<4pt>{}="A",
\ar@{-} "A"; (20,0)   *++!D{6}  *\cir<4pt>{}="B",
\ar@{-} "B"; (30,0)   *+!L+!D{8}*\cir<4pt>{}="C",
\ar@{-} "C"; (40,0)   *++!D{6}  *\cir<4pt>{}="D",
\ar@{-} "D"; (50,0)   *++!D{4}  *\cir<4pt>{}="E",
\ar@{-} "E"; (60,0)   *++!D{2}  *{\cdot}*\cir<4pt>{}="F",
\ar@{-} "F"; (70,0)   *++!D{1}  *\cir<4pt>{}.
\ar@{-} "C"; (30,10)  *++!L{4}  *\cir<4pt>{},
\end{xy}\allowdisplaybreaks\\
\begin{xy}
\ar@{-}               *++!D{1}   *{\cdot}*\cir<4pt>{};
             (10,0)   *++!D{4}   *\cir<4pt>{}="A",
\ar@{-} "A"; (20,0)   *++!D{7}   *\cir<4pt>{}="B",
\ar@{-} "B"; (30,0)   *+!L+!D{10}*\cir<4pt>{}="C",
\ar@{-} "C"; (40,0)   *++!D{8}   *\cir<4pt>{}="D",
\ar@{-} "D"; (50,0)   *++!D{6}   *\cir<4pt>{}="E",
\ar@{-} "E"; (60,0)   *++!D{4}   *\cir<4pt>{}="F",
\ar@{-} "F"; (70,0)   *++!D{2}   *\cir<4pt>{}.
\ar@{-} "C"; (30,10)  *++!L{5}   *\cir<4pt>{},
\end{xy}\allowdisplaybreaks\\
\begin{xy}
\ar@{-}               *++!D{1}  *\cir<4pt>{};
             (10,0)   *++!D{2}  *\cir<4pt>{}="A",
\ar@{-} "A"; (20,0)   *++!D{3}  *\cir<4pt>{}="B",
\ar@{-} "B"; (30,0)   *++!D{4}  *\cir<4pt>{}="C",
\ar@{-} "C"; (40,0)   *+!L+!D{5}*\cir<4pt>{}="D",
\ar@{-} "D"; (50,0)   *++!D{4}  *\cir<4pt>{}="E",
\ar@{-} "E"; (60,0)   *++!D{3}  *\cir<4pt>{}="F",
\ar@{-} "F"; (70,0)   *++!D{2}  *\cir<4pt>{}="G".
\ar@{-} "G"; (80,0)   *++!D{1}  *\cir<4pt>{}.
\ar@{-} "D"; (40,10)  *++!L{2}  *{\cdot}*\cir<4pt>{},
\end{xy}\allowdisplaybreaks\\
\begin{xy}
\ar@{-}               *++!D{1}  *\cir<4pt>{};
             (10,0)   *++!D{2}  *{\cdot}*\cir<4pt>{}="A",
\ar@{-} "A"; (20,0)   *++!D{4}  *\cir<4pt>{}="B",
\ar@{-} "B"; (30,0)   *+!L+!D{6}*\cir<4pt>{}="C",
\ar@{-} "C"; (40,0)   *++!D{5}  *\cir<4pt>{}="D",
\ar@{-} "D"; (50,0)   *++!D{4}  *\cir<4pt>{}="E",
\ar@{-} "E"; (60,0)   *++!D{3}  *\cir<4pt>{}="F",
\ar@{-} "F"; (70,0)   *++!D{2}  *\cir<4pt>{}="G".
\ar@{-} "G"; (80,0)   *++!D{1}  *\cir<4pt>{}.
\ar@{-} "C"; (30,10)  *++!L{3}  *\cir<4pt>{}
\end{xy}\allowdisplaybreaks\\
\begin{xy}
\ar@{-}               *++!D{4}   *\cir<4pt>{};
             (10,0)   *++!D{8}   *\cir<4pt>{}="A",
\ar@{-} "A"; (20,0)   *+!L+!D{12}*\cir<4pt>{}="B",
\ar@{-} "B"; (30,0)   *++!D{10}  *\cir<4pt>{}="C",
\ar@{-} "C"; (40,0)   *++!D{8}   *\cir<4pt>{}="D",
\ar@{-} "D"; (50,0)   *++!D{6}   *\cir<4pt>{}="E",
\ar@{-} "E"; (60,0)   *++!D{4}   *\cir<4pt>{}="F",
\ar@{-} "F"; (70,0)   *++!D{2}   *{\cdot}*\cir<4pt>{}="G".
\ar@{-} "G"; (80,0)   *++!D{1}   *\cir<4pt>{}.
\ar@{-} "B"; (20,10)  *++!L{6}   *\cir<4pt>{},
\end{xy}\allowdisplaybreaks\\
\begin{xy}
\ar@{-}               *++!D{2}  *{\cdot}*\cir<4pt>{};
             (10,0)   *++!D{5}  *\cir<4pt>{}="A",
\ar@{-} "A"; (20,0)   *+!L+!D{8}*\cir<4pt>{}="B",
\ar@{-} "B"; (30,0)   *++!D{7}  *\cir<4pt>{}="C",
\ar@{-} "C"; (40,0)   *++!D{6}  *\cir<4pt>{}="D",
\ar@{-} "D"; (50,0)   *++!D{5}  *\cir<4pt>{}="E",
\ar@{-} "E"; (60,0)   *++!D{4}  *\cir<4pt>{}="F",
\ar@{-} "F"; (70,0)   *++!D{3}  *\cir<4pt>{}="G".
\ar@{-} "G"; (80,0)   *++!D{2}  *\cir<4pt>{}="H".
\ar@{-} "H"; (90,0)   *++!D{1}  *\cir<4pt>{}.
\ar@{-} "B"; (20,10)  *++!L{4}  *\cir<4pt>{},
\end{xy}
\end{gather*}}

\section{Connection problem}\label{sec:C}
Fix a tuple $\mathbf m=\bigl(m_{j,\nu}\bigr)_{\substack{j=0,\dots,k\\ \nu=1,\dots,n_j}}\in\mathcal P_{k+1}^{(n)}$ in this section.
For complex numbers $\lambda_{j,\nu}\in\mathbb C$ and $\mu\in\mathbb C$ we put
\begin{align*}
 \{\lambda_{\mathbf m}\}&:=
 \begin{Bmatrix}
   [\lambda_{0,1}]_{(m_{0,1})}     
     &\cdots & [\lambda_{k,1}]_{(m_{k,1})}\\
    \vdots                         
   &\vdots & \vdots\\
   [\lambda_{0,n_0}]_{(m_{0,n_0})} 
     &\cdots & [\lambda_{k,n_k}]_{(m_{k,n_k})}
 \end{Bmatrix},\quad
 [\mu]_{(p)}:=
  \begin{pmatrix}
   \mu\\ \mu+1\\ \vdots\\ \mu+p-1
  \end{pmatrix}.
\end{align*}
We may identify $\{\lambda_{\mathbf m}\}$ with an element of 
$M(n,k+1,\mathbb C)$.
\begin{dfn}
A rigid tuple $\mathbf m\in\mathcal R_{k+1}$ 
is a \textsl{rigid sum} of $\mathbf m'$ and $\mathbf m''$ if
\begin{equation}
  \mathbf m=\mathbf m'+\mathbf m''\text{ \ and \ }
  \mathbf m',\ \mathbf m''\in\mathcal R_{k+1}
\end{equation}
and we express this by $\mathbf m=\mathbf m'\oplus\mathbf m''$,
which we call a \textsl{rigid decomposition} of $\mathbf m$.
\end{dfn}
\begin{thm}\label{thm:c}
Fix $k+1$ points 
$\{z_0,\dots,z_k\} \subset\mathbb C\cup\{\infty\}$
and a rigid tuple $\mathbf m\in\mathcal R_{k+1}$.
Assume $\lambda_{j,\nu}\in\mathbb C$ are generic under the Fuchs relation
$|\{ \lambda_{\mathbf m}\}|=0$ with
\begin{align}
 |\{ \lambda_{\mathbf m}\}|&
   :=\sum_{j=0}^k\sum_{\nu=0}^{n_j}m_{j,\nu}\lambda_{j,\nu}
    - \ord\mathbf m+1.
\end{align}

{\rm i)\ } 
There uniquely exists a single Fuchsian differential equation $Pu=0$ 
of order $n$ with regular singularities at 
$\{z_0,\dots,z_k\}\subset\mathbb C\cup\{\infty\}$
such that the set of exponents at $z_j\subset\mathbb C\cup\{\infty\}$ 
is equal to that of components of the $(j+1)$-th column of 
$\{\lambda_{\mathbf m}\}$ and moreover that the local monodromies 
are semisimple at $z_j$ for $j=0,\dots,k$.

\smallskip
{\rm ii)\ }
Assume $k=2$, $m_{0,n_0}=m_{1,n_1}=1$ and
$m_{j,\nu} > 0$ for $\nu=1,\ldots,n_j$ and $j=0,1,2$.
Let $c(\lambda_{0,n_0}\!\rightsquigarrow\!\lambda_{1,n_1})$ 
denote the connection coefficient from the normalized local solution 
of $Pu=0$ in i) corresponding to the exponent $\lambda_{0,n_0}$ at $z_0$
to the normalized local solution corresponding to 
the exponent $\lambda_{1,n_1}$ at $z_1$. Then
\begin{align}\label{eq:connection}
 c(\lambda_{0,n_0}\!\rightsquigarrow\!\lambda_{1,n_1})
 &=\frac
  {\displaystyle\prod_{\nu=1}^{n_0-1} 
    \Gamma\bigl(\lambda_{0,n_0}-\lambda_{0,\nu}+1\bigr)
   \cdot\prod_{\nu=1}^{n_1-1}
    \Gamma\bigl(\lambda_{1,\nu}-\lambda_{1,n_1}\bigr)
  }
  {\displaystyle\prod_{\substack{\mathbf m'\oplus\mathbf m''=\mathbf m\\
                    m'_{0,n_0}=m''_{1,n_1}=1}}
    \Gamma\bigl(|\{\lambda_{\mathbf m'}\}|\bigr)
  },\allowdisplaybreaks\\
  \sum_{\substack{\mathbf m'\oplus\mathbf m''=\mathbf m\\
    m'_{0,n_0}=m''_{1,n_1}=1}}
   \!\!\!\!\!\! m'_{j,\nu}
  &= (n_1-1)m_{j,\nu}-\delta_{j,0}(1-n_0\delta_{\nu,_{n_0}})
   +\delta_{j,1}(1-n_1\delta_{\nu,_{n_1}})\label{eq:concob}\\[-.5cm]
  &\hspace{4.5cm}(0\le j\le 2,\ 1\le\nu\le n_j).\notag
\end{align}
\end{thm}
\begin{rem} 
{\rm i)\ } Putting $(j,\nu)=(0,n_0)$ in \eqref{eq:concob} or
considering the sum $\sum_\nu$ for \eqref{eq:concob} with $j=1$, we have
\begin{gather}
 \#\{\mathbf m'\in\mathcal R_3\,;\,\mathbf m'\oplus\mathbf m''=\mathbf m\text{ \ with \ }
m'_{0,n_0}=m''_{0,n_1}=1\} = n_0+n_1-2,\label{eq:numdec}\\
 \sum_{\substack{\mathbf m'\oplus\mathbf m''=\mathbf m\\
    m'_{0,n_0}=m''_{1,n_1}=1}}
   \!\!\!\!\!\! \ord\mathbf m'=(n_1-1)\ord\mathbf m.
\end{gather}

{\rm ii)\ } We may regard $\{\lambda_{\mathbf m}\}$ as a Riemann scheme
of the Fuchsian equation with the condition that the local monodromies at
the singular points are semisimple for generic $\lambda_{j,\nu}$ under the
Fuchs condition. The equation for general $\lambda_{j,\nu}$ is
defined by the analytic continuation.
The corresponding Riemann scheme will be denoted by $P\{\lambda_{\mathbf m}\}$.

\smallskip
{\rm iii)\ }
A proof of this theorem and related results will be given
in another paper.
The proof is a generalization of that of Gauss summation formula for 
Gauss hypergeometric series due to Gauss, which doesn't use integral 
representations of the solutions.

\smallskip
{\rm iv)\ }  In the theorem the condition $k=2$ means that there exists
no geometric moduli in the Fuchsian equation and
we may assume $(z_0,z_1,z_2)=(0,1,\infty)$.  
By the transformation of the solutions 
$u\mapsto z^{-\lambda_{0,n_0}}(1-z)^{-\lambda_{1,n_1}}u$
we may moreover assume $\lambda_{0,n_0}=\lambda_{1,n_1}=0$.
Then the meaning of ``normalized local solution" is clear
under the condition $m_{0,n_0}=m_{1,n_1}=1$.

\smallskip
{\rm v)\ } By the aid of a computer the author obtained the table of the 
concrete connection coefficients \eqref{eq:connection} for 
$\mathbf m\in\mathcal R_3$ satisfying $\ord\mathbf m\le 40$ 
together with checking \eqref{eq:concob}, 
which contains 4,111,704 independent cases.
\end{rem}
\begin{exa}[$H_n:$ hypergeometric family]
The Fuchsian differential equation of hypergeometric family
of order $n$ has the spectral type $\mathbf m=(1^n,n-11,1^n)$.
Its Riemann scheme is 
\begin{equation}
P\begin{Bmatrix}
     \lambda_{0,1} & \ [\lambda_{1,1}]_{(n-1)}\   & \lambda_{2,1}\\
     \vdots        &              & \vdots \\
     \lambda_{0,n-1}&  &\lambda_{2,n-1}\\
     \lambda_{0,n}  & \lambda_{1,2} &\lambda_{2,n}
    \end{Bmatrix}
\end{equation}
with complex numbers $\lambda_{j,\nu}$ satisfying the Fuchs relation
\begin{equation}
  \sum_{\nu}(\lambda_{0,\nu}+\lambda_{2,\nu})
  +(n-1)\lambda_{1,1}+\lambda_{1,2}=n-1.
\end{equation}
It follows from \eqref{eq:numdec} that there are $n$ rigid 
decompositions $\mathbf m=\mathbf m'\oplus\mathbf m''$ of $\mathbf m$ with 
$m'_{0,n}=m''_{1,2}=1$ and they are\\[-15pt]
\begin{align*}
1\cdots1\overline{1}\,,\,n-1\underline{1}\,,\,1\cdots1 &=
0\cdots0\overline{1}\,,\,\ \ 1\ \ \ \,\underline{0}\,,\,0\cdots0\overset{\underset{\smallsmile}i}10\cdots0\\[-5pt]
    &\,\oplus 1\cdots1\overline{0}\,,\,n-2\underline{1}\,,\,1\cdots101\cdots1
    \qquad(i=1,\dots,n),
\end{align*}
which are symbolically expressed by $H_n=H_1\oplus H_{n-1}$.
Then the formula \eqref{eq:connection} implies
\begin{align}
c(\lambda_{0,n}\rightsquigarrow\lambda_{1,2})
&=\frac{\displaystyle\prod_{i=1}^{n-1}
  \Gamma({\lambda_{0,n}}-\lambda_{0,i}+1)
   \cdot\Gamma(\lambda_{1,1}-{\lambda_{1,2}})}
  {\displaystyle\prod_{i=1}^n\Gamma({\lambda_{0,n}}+\lambda_{1,1}+\lambda_{2,i})}
\intertext{and}\notag\\[-38pt]
  c(\lambda_{1,2}\rightsquigarrow\lambda_{0,n})
 &=\frac{\Gamma\bigl(\lambda_{1,2}-\lambda_{1,1}+1\bigr)\cdot
   \displaystyle\prod_{i=1}^{n-1}\Gamma\bigl(\lambda_{0,i}-\lambda_{0,n}
  \bigr)}
  {\displaystyle\prod_{i=1}^n
  \Gamma\bigl(
   \left|\begin{Bmatrix}
    (\lambda_{0,\nu})_{1\le\nu\le n-1}
   & \ [\lambda_{1,1}]_{(n-2)}\  & 
   & (\lambda_{2,\nu})_{\substack{1\le\nu\le n\\ \nu\ne i}}\\[-4pt]
   & \lambda_{1,2}
   \end{Bmatrix}\right|
%
  \bigr)}.
\end{align}
Here we denote
\[
  (\mu_\nu)_{1\le\nu\le n}
  =\left(\begin{smallmatrix}\mu_1\\ \mu_2\\ \vdots\\ \mu_n\end{smallmatrix}\right)\in\mathbb C^n
  \text{\ \ and \ \ }
  (\mu_\nu)_{\substack{1\le\nu\le n\\ \nu\ne i}}
  =\left(\begin{smallmatrix}\mu_1\\ \vdots\\ \mu_{i-1}\\ \mu_{i+1}\\ \vdots\\ \mu_n\
   \end{smallmatrix}\right)\in\mathbb C^{n-1}
\]
for complex numbers $\mu_1,\dots,\mu_n$.
In the same way as above we have 
\begin{equation}
\begin{split}
c(\lambda_{0,n}\rightsquigarrow\lambda_{2,n})
&=\prod_{i=1}^{n-1}\frac{
    \Gamma(\lambda_{2,i}-\lambda_{2,n})}
  {\Gamma\bigl(
   \left|\begin{Bmatrix}
     \lambda_{0,n}
   &\ \lambda_{1,1}\ 
   & \lambda_{2,i}
   \end{Bmatrix}\right|
   \bigr)}\\
 &\quad\cdot\prod_{i=1}^{n-1}\frac{
    \Gamma(\lambda_{0,n}-\lambda_{0,i}+1)}
  {\Gamma(
   \left|\begin{Bmatrix}
    (\lambda_{0,\nu})_{\substack{1\le\nu\le n\\ \nu\ne i}}
   &\ [\lambda_{1,1}]_{(n-2)}\ & 
   & (\lambda_{2,\nu})_{1\le \nu\le n-1})\\[-4pt]
      & \lambda_{1,2}
   \end{Bmatrix}\right|
   )}
\end{split}
\end{equation}
by the rigid decompositions\\[-17pt]
\begin{align*}
 1\cdots1\overline{1}\,,\,n-11\,,\,1\cdots1\underline{1}
 &\!=0\cdots0\overline{1}\,,\,\ \ \ 1\,\ \ 0\,,\,0\ldots0\overset{\underset{\smallsmile}i}10\cdots0\underline{0}
\\[-8pt]
 &\oplus1\cdots1\overline{0}\,,\,n-21\,,\,1\cdots101\cdots1\underline{1}\\
 &\!=1\cdots1\overset{\underset\smallsmile i}01\cdots1\overline{1}\,,\,n-21\,,\,1\cdots1\underline{0}
\\[-8pt]
 &\oplus0\ldots010\cdots0\overline{0}\,,\,\ \ \ 1\ \ \,0\,,\,0\cdots0\underline{1}\qquad(i=1,\dots,n-1).
\end{align*}

The generalized hypergeometric series
\[
  {}_nF_{n-1}(\alpha_1,\dots,\alpha_n,\beta_1,\dots,\beta_{n-1};z) 
  = \displaystyle\sum_{k=0}^\infty
 \frac{(\alpha_1)_{k}\cdots(\alpha_n)_k}
                        {(\beta_1)_k\cdots(\beta_{n-1})_k(1)_k}z^k
\]
is a solution of the differential equation
\[
\Bigl(\prod_{j=1}^{n-1}(z\frac d{dz}+\beta_j)\cdot\frac d{dz}
  -\prod_{j=1}^{n}(z\frac d{dz}+\alpha_j)\Bigr)u=0
\]
with the Riemann scheme
\[
P\begin{Bmatrix}
{z=0} &{1}&{\infty}\\
 1-\beta_1     & \ [0]_{(n-1)}\  &\alpha_1\\
 \vdots        &          &\vdots & {;\ z}\\
 1-\beta_{n-1} &              &\alpha_{n-1}\\
 0 & -\beta_n   &\alpha_n
\end{Bmatrix}
\quad\text{with \ }
\sum_{\nu=1}^n \alpha_\nu = \sum_{\nu=1}^n \beta_\nu.
\]
This is the Gauss hypergeometric sereis when $n=2$. Here we denote
\[
  (\gamma)_k=\gamma(\gamma+1)\cdots(\gamma+k-1)
\]
for $\gamma\in\mathbb C$ and $k=1,2,\ldots$ and $(\gamma)_0=1$.
Hence by putting
\begin{gather*}
 \lambda_{0,\nu}=1-\beta_\nu\ \ (1\le\nu\le n-1),\ \ \lambda_{0,n}=0,\ \ 
 \lambda_{1,1}=0,\\ \lambda_{1,2}=-\beta_n
 \text{ \ and \ }\lambda_{2,i}=\alpha_i\ \ (1\le i\le n)
\end{gather*}
we have
\begin{align*}
c(\lambda_{0,n}\rightsquigarrow\lambda_{1,2})
&=\displaystyle\prod_{i=1}^n\frac{\Gamma(\beta_i)}
  {\Gamma(\alpha_i)}\\
&=\lim_{x\to 1-0}(1-x)^{\beta_n}
  {}_nF_{n-1}(\alpha,\beta;x)
  \qquad(\operatorname{Re} \beta_n > 0),\\
  c(\lambda_{1,2}\rightsquigarrow\lambda_{0,n})
 &=\prod_{i=1}^n\frac{\Gamma(1-\beta_i)}{\Gamma(1-\alpha_i)},
\quad
c(\lambda_{0,n}\rightsquigarrow\lambda_{2,n})
=\prod_{i=1}^{n-1}\frac{\Gamma(\beta_i)\Gamma(\alpha_i-\alpha_n)}
   {\Gamma(\alpha_i)\Gamma(\beta_i-\alpha_n)}.
\end{align*}
These connection coefficients are calculated by Levelt \cite{Le} 
and Okubo et al \cite{OTY}.
\end{exa}
\begin{exa}[$EO_{2m}:$ even family]
Let $m$ be a positive integer.
The single Fuchsian differential equation whose Riemann scheme is
\begin{equation}
P\begin{Bmatrix}
     \lambda_{0,1} & [\lambda_{1,1}]_{(m)}   & [\lambda_{2,1}]_{(m)}\\
     \vdots        &\ [\lambda_{1,2}]_{(m-1)}\  & [\lambda_{2,2}]_{(m)}\\
     \lambda_{0,2m}& \lambda_{1,3}
    \end{Bmatrix}
\end{equation}
with the Fuchs relation
\begin{equation}
\sum_{\nu=1}^{2m}\lambda_{0,\nu}
+m\lambda_{1,1}+(m-1)\lambda_{1,2}+\lambda_{1,3}
+m\lambda_{2,1}+m\lambda_{2,2}
=2m-1
\end{equation}
is of even family of order $2m$.
Then Theorem~\ref{thm:c} and the rigid decompositions
\begin{align*}
 1\cdots1\overline1\,,\,mm-1\underline1\,,\,mm
  &=0\cdots0\overline1\,,\,10\underline0\,,\,
    \overset{\underset{\smallsmile} i}10\oplus
    1\cdots1\overline0\,,\,m-1m-1\underline1\,,
    \,\overset{\underset{\smallsmile} i}01\\[-3pt]
  &=0\cdots\overset{\underset{\smallsmile} j}1
    \overline1\,,\,11\underline0\,,\,11\oplus
    1\cdots\overset{\underset{\smallsmile} j}0
    \overline0\,,\,m-1m-2\underline1\,,\,m-1m-1,
\end{align*}
which are symbolically expressed by 
$EO_{2m}=H_1\oplus EO_{2m-1}=H_2\oplus EO_{2m-2}$,
imply
\begin{align*}
 c(\lambda_{0,2m}\rightsquigarrow\lambda_{1,3})
 &=\prod_{i=1}^2\frac
 {\Gamma\bigl(\lambda_{1,i}-\lambda_{1,3}\bigr)
  }{\Gamma\bigl(
    \left|\begin{Bmatrix}
   \lambda_{0,2m} &\ \lambda_{1,1}\ &\lambda_{2,i}
    \end{Bmatrix}\right|
  \bigr)}
  \cdot
   \prod_{j=1}^{2m-1}\frac
  {\Gamma\bigl(\lambda_{0,2m}-\lambda_{0,j}+1)}
  {\Gamma\bigl(
    \left|\begin{Bmatrix}
    \lambda_{0,j} & \lambda_{1,1} & \lambda_{2,1}\\
    \lambda_{0,2m} &\ \lambda_{1,2}\ & \lambda_{2,2}
  \end{Bmatrix}\right|
   \bigr)},\allowdisplaybreaks\\
 c(\lambda_{1,3}\rightsquigarrow\lambda_{0,2m})
 &=\displaystyle\prod_{i=1}^2\frac
 {\Gamma\bigl(\lambda_{1,3}-\lambda_{1,i}+1\bigr)
  }{\Gamma\bigl(
   \left|\begin{Bmatrix}
    &[\lambda_{1,1}]_{(m-1)}&[\lambda_{2,\nu}]_{(m)}\\
    (\lambda_{0,\nu})_{1\le\nu\le 2m-1}
   &\ [\lambda_{1,2}]_{(m-1)}\ &[\lambda_{2,3-i}]_{(m-1)}
     \\
    &\lambda_{1,3}
  \end{Bmatrix}\right|
  \bigr)}
  \\
 &\quad
  \cdot
   \prod_{j=1}^{2m-1}\frac
  {\Gamma\bigl(\lambda_{0,j}-\lambda_{0,2m})}
  {\Gamma\bigl(
     \left|\begin{Bmatrix}
   & [\lambda_{1,1}]_{(m-1)}&[\lambda_{2,1}]_{(m-1)}\\
   (\lambda_{0,\nu})_{\substack{1\le\nu\le 2m-1\\ \nu\ne j}}
  &\ [\lambda_{1,2}]_{(m-2)}\ &[\lambda_{2,2}]_{(m-1)}\\[-4pt]
  &\lambda_{1,3}&
  \end{Bmatrix}\right|
   \bigr)}.
\end{align*}
\end{exa}

\section{Appendix}\label{sec:apd}
Crawley-Boevey \cite{CB} gives the following complete answer to 
the additive Deligne-Simpson problem.
\begin{thm}[\cite{CB}]\label{thm:CB}
Let $k$ and $n$ be positive integers, 
$\mathbf m_j=(m_{j,1},\dots,m_{j,n_j})$ be partitions of $n$ and
$\lambda_j = (\lambda_{j,1},\dots,\lambda_{j,n_j})\in \mathbb C^{n_j}$
for $j=0,\dots,k$.
Put $\mathbf m=(\mathbf m_0,\dots,\mathbf m_k)\in\mathcal P^{(n)}_{k+1}$ and 
assume the condition \eqref{eq:sum0}.
Then there exists an irreducible tuple of matrices 
$\mathbf A=(A_0,\dots,A_k)\in M(n,\mathbb C)^{k+1}$ satisfying 
\begin{equation}\label{eq:DSCB}
  A_j\sim L(\mathbf m_j;\lambda_j)\quad(j=0,\dots,k)\text{ \ and \ }
  A_0+\cdots+A_k=0
\end{equation}
if and only if $\alpha_{\mathbf m}$ is a positive root and moreover
\begin{equation}\label{eq:CCB}
  \Bigl(\sum_{j,\,\nu}m_{j,\nu}^{(1)}\lambda_{j,\nu},\dots,
   \sum_{j,\,\nu}m_{j,\nu}^{(N)}\lambda_{j,\nu}\Bigr)\ne (0,\dots,0)\in\mathbb C^N
\end{equation}
for any decomposition
\begin{equation}\label{eq:CBDec}
  \mathbf m = \mathbf m^{(1)}+\cdots+\mathbf m^{(N)}
\end{equation}
with $N\ge 2$ and $\mathbf m^{(i)}\in\mathcal P_{k+1}$ such that
\begin{equation}\label{eq:Pdec}
 \begin{cases}
 \alpha_{\mathbf m^{(i)}}\text{ defined by \eqref{eq:Kazpart} 
  are positive roots \  $(i=1,\dots,N)$,}\\
 \Pidx\mathbf m\le \Pidx\mathbf m^{(1)}+\cdots+\Pidx\mathbf m^{(N)}
 \end{cases}
\end{equation}
under the notation and the correspondence in Remark~\ref{rem:KM} i).
\end{thm}
K.~Takemura indicated to the author that the following result follows from
Theorem~\ref{thm:CB} and kindly allows the author to include the proof in this 
note.
\begin{thm}\label{thm:GDS}
Retain the notation and the assumption in Theorem~\ref{thm:CB}.
If there exists an irreducible tuple of matrices 
$\mathbf A=(A_0,\dots,A_k)\in M(n,\mathbb C)^{k+1}$ satisfying \eqref{eq:DSCB},
then $\alpha_{\mathbf m}$ defined by \eqref{eq:Kazpart} is a positive root 
such that $\mathbf m$ is indivisible or $\idx\mathbf m<0$.
Conversely if a tuple $\mathbf m\in\mathcal P$ is indivisible or $\mathbf m$
satisfies $\idx\mathbf m<0$ and moreover $\alpha_{\mathbf m}$ is a 
positive root, then $\mathbf m$ is irreducibly realizable.
\end{thm}
\begin{proof}
Note that this theorem follows from Theorem~\ref{thm:CB} 
if $\mathbf m$ is indivisible because \eqref{eq:CCB} always holds
when $\lambda_{j,\nu}$ are generic under the condition \eqref{eq:sum0}.

Suppose $\mathbf m=d\overline{\mathbf m}$ with an integer $d>1$ and 
an indivisible tuple $\overline{\mathbf m}\in\mathcal P_{k+1}$.
Since $\Pidx\mathbf m=1-\frac12\idx\mathbf m=
1-\frac12(\alpha_{\mathbf m},\alpha_{\mathbf m})$,
we have
\begin{equation}\label{eq:Midx}
 \Pidx d\overline{\mathbf m}=1+d^2(\Pidx\overline{\mathbf m}-1).
\end{equation}
If $\Pidx\overline{\mathbf m}=1$, we have
$\Pidx{\mathbf m}=\Pidx\,(d-1)\overline{\mathbf m}=1$
and this theorem also follows from Theorem~\ref{thm:CB} with
the decomposition $\mathbf m=\overline{\mathbf m}+(d-1)\overline{\mathbf m}$
corresponding to \eqref{eq:CBDec}.

Hence we may moreover suppose $\Pidx\overline{\mathbf m}>1$.
Assume the existence of the decomposition \eqref{eq:CBDec} such that
$\sum_{j,\nu}m_{j,\nu}^{(i)}\lambda_{j,\nu}=0$ in Theorem~\ref{thm:CB}.
If $\lambda_{j,\nu}$ are generic, we have 
$\mathbf m^{(i)}=d_i\overline{\mathbf m}$
with positive integers $d_i$ satisfying $d=d_1+\cdots+d_N$.
Then 
\begin{align*}
\Pidx\mathbf m-\sum_{i=1}^N \Pidx d_i\overline{\mathbf m}&=
1 + d^2(\Pidx\overline{\mathbf m}-1) - 
 \sum_{i=1}^N\bigl(1+d_i^2(\Pidx\overline{\mathbf m}-1)\bigr)\\
&=2\sum_{1\le i<j\le N}d_id_j(\Pidx\overline{\mathbf m}-1) - (N-1) > 0
\end{align*}
when $\Pidx\overline{\mathbf m}\ge 2$ and $N\ge 2$.
Hence Theorem~\ref{thm:CB} completes the proof. 
\end{proof}
\begin{rem} \textrm{i)} \ 
Kostov \cite{Ko2} studies the above result when $\idx\mathbf m=0$.

\textrm{ii)}\,
It follows from Theorem~\ref{thm:GDS} that 
the spectral type of any irreducible tuple 
$\mathbf A\in M(n,\mathbb C)_0^{k+1}$ is irreducibly realizable.

\textrm{iii)}
We define that a tuple
$\mathbf A\in M(n,\mathbb C)_0^{k+1}$ and the corresponding 
Fuchsian system \eqref{eq:Fuchs} are \textsl{fundamental}
if $\mathbf A$ is irreducible and cannot be transformed into a tuple of
matrices with a lower rank by any successive applications of additions and 
middle convolutions.
We also define that a tuple $\mathbf m\in\mathcal P$ is \textsl{fundamental} if
it corresponds to a suitable fundamental tuple $\mathbf A\in M(n,\mathbb C)_0^{k+1}$.

Then a tuple $\mathbf m\in\mathcal P$ is {fundamental} if and only if 
$\mathbf m$ is basic or there exist a positive number $d$ and a basic tuple 
$\overline{\mathbf m}\in\mathcal P$ satisfying 
$\mathbf m=d\overline{\mathbf m}$ and $\idx\overline{\mathbf m}<0$.

Hence it follows from Proposition~\ref{prop:basic} and the equality 
\eqref{eq:Midx} that there exist only a finite number of fundamental tuples 
$\mathbf m\in\mathcal P$ such that $\idx\mathbf m$ equal to 
a fixed number.

\textrm{iv) } (Nilpotent case : \cite{CB}, \cite{Ko3})
Under the notation in Theorem~\ref{thm:CB} there exists an irreducible tuple 
$\mathbf A\in M(n,\mathbb C)^{k+1}$ satisfying \eqref{eq:DSCB} with
$\lambda_{j,\nu}=0$ for any $j$ and $\nu$ if and only if $\ord\mathbf m=1$ or 
$\mathbf m$ is fundamental and moreover 
$\mathbf m$ is not the special element in Example~\ref{ex:special}
with $m\ge 2$.
Here we have the decompositions $D_4^{(m+1)}=D_4^{(m)}+D_4^{(1)}$ and
$E_j^{(m+1)}=E_j^{(m)}+E_j^{(1)}$ for $j=6$, $7$ and $8$ which satisfy
\eqref{eq:CBDec} and \eqref{eq:Pdec}.
\end{rem}

\end{document}